\documentclass[a4paper]{article}
\usepackage[T1]{fontenc}
\usepackage{lmodern}
\usepackage{amsmath,amssymb}
\usepackage{amsthm}
\usepackage{mathrsfs}
\usepackage[all]{xy}
\usepackage{rotate}
\usepackage{enumitem}
\usepackage{mathtools}
\usepackage[dvipdfmx]{graphics}

\usepackage{here}

\theoremstyle{plain}
\newtheorem{theorem}{Theorem}[section]
\newtheorem*{theorem*}{Theorem}
\newtheorem{proposition}[theorem]{Proposition}
\newtheorem{lemma}[theorem]{Lemma}
\newtheorem{corollary}[theorem]{Corollary}

\newtheorem{fact}[theorem]{Fact}

\theoremstyle{definition}
\newtheorem{definition}[theorem]{Definition}

\newtheorem{remark}[theorem]{Remark}

\numberwithin{equation}{subsection}

\newcommand{\define}[1]{\textit{#1}}

\newcommand{\set}[1]{\left\{#1 \right\}}

\newcommand{\Coker}{\mathrm{Coker}}
\newcommand{\Ker}{\mathrm{Ker}}

\newcommand{\Hom}{\mathrm{Hom}}
\DeclareMathOperator{\End}{End}
\DeclareMathOperator{\Aut}{Aut}
\DeclareMathOperator{\id}{id}
\newcommand{\cxdot}{\bullet}
\newcommand{\cpx}[1]{#1^{\cxdot}}

\newcommand{\BasicRing}[1]{\mathbb{#1}}
\newcommand{\ZZ}{\BasicRing{Z}}
\newcommand{\NN}{\BasicRing{N}}
\newcommand{\QQ}{\BasicRing{Q}}

\newcommand{\CC}{\BasicRing{C}}

\newcommand{\Mod}{\mathrm{Mod}}
\newcommand{\Ann}{\mathrm{Ann}}

\newcommand{\Tor}{\mathrm{Tor}}

\newcommand{\opalg}{\mathrm{op}}
\newcommand{\alg}[1]{\mathcal{#1}}

\DeclareMathOperator{\gr}{gr}
\newcommand{\Ass}{\mathrm{Ass}}

\newcommand{\classicalG}[1]{\mathrm{#1}}

\newcommand{\lieC}[1]{\mathfrak{#1}}
\newcommand{\lie}[1]{\mathfrak{#1}}

\newcommand{\univ}[1]{\mathcal{U}(\lieC{#1})}
\newcommand{\univcent}[1]{\mathcal{Z}(\lieC{#1})}

\newcommand{\LieSL}{\classicalG{SL}}

\newcommand{\Dsheaf}[2]{\mathscr{#1}_{#2}}
\newcommand{\ntDsheaf}[1]{\Dsheaf{D}{#1}}
\newcommand{\Dalg}[2]{\mathcal{#1}_{#2}}
\newcommand{\ntDalg}[1]{\Dalg{D}{#1}}

\newcommand{\sect}{\Gamma}
\DeclareMathOperator{\Ch}{Ch}

\newcommand{\locZuck}[2]{\mathrm{L}^{#1}_{#2}}
\newcommand{\DlocZuck}[3]{\mathbb{D}^{#3}\locZuck{#1}{#2}}
\newcommand{\zuck}[2]{\Gamma^{#1}_{#2}}
\newcommand{\Dzuck}[3]{\mathbb{D}^{#3}\zuck{#1}{#2}}

\newcommand{\Dmod}[1]{\mathcal{#1}}

\DeclareMathOperator{\loc}{L}

\newcommand{\regular}{\mathcal{O}}
\newcommand{\rring}[1]{\regular(#1)}
\newcommand{\rsheaf}[1]{\regular_{#1}}
\newcommand{\invshf}{\vee}
\DeclareMathOperator{\Variety}{\mathcal{V}}

\newcommand{\roots}{\Delta}
\newcommand{\proots}{\Delta^+}

\newcommand*{\longw}{w_l}
\DeclareMathOperator{\GKdim}{GKdim}

\DeclareMathOperator{\Rrankmax}{loc\text{.}\overline{rank}}
\newcommand{\rfilt}[2]{\mathcal{R}_{#1}(#2)}

\begin{document}

\title{$\mathscr{D}$-modules on the basic affine space and large $\mathfrak{g}$-modules}

\author{Masatoshi Kitagawa}

\date{}

\maketitle

\begin{abstract}
	In this paper, we treat $\mathscr{D}$-modules on the basic affine space $G/U$ and their global sections for a semisimple complex algebraic group $G$.
	Our aim is to prepare basic results about large non-irreducible modules for the branching problem and harmonic analysis of reductive Lie groups.
	A main tool is a formula given by Bezrukavnikov--Braverman--Positselskii.
	The formula is about a product of functions and their Fourier transforms on $G/U$ like Capelli's identity.
	Using the formula, we give a generalization of the Beilinson--Bernstein correspondence.

	We show that the global sections of holonomic $\mathscr{D}$-modules are also holonomic using the formula.
	As a consequence, we give a large algebra action on the $\mathfrak{u}$-cohomologies $H^i(\mathfrak{u}; V)$ of a $\mathfrak{g}$-module $V$
	when $V$ is realized as a holonomic $\mathscr{D}$-module.
	We consider affinity of the supports of the $\mathfrak{t}$-modules $H^i(\mathfrak{u}; V)$.
\end{abstract}

\section{Introduction}

In the branching problem and harmonic analysis of reductive Lie groups, we need to treat large non-irreducible modules of a reductive Lie algebra $\lie{g}$.
Such $\lie{g}$-modules often have some additional structure, e.g.\ action of a larger (Lie) algebra and geometric realization.
In this paper, we deal with the special $G$-variety $G/U$ called the basic affine space,
and $\ntDsheaf{}$-modules on $G/U$, which we consider as `good' large $\lie{g}$-modules.
Our main tool is a formula proved by Bezrukavnikov--Braverman--Positselskii in \cite[Proposition 3.11]{BBP02}.
Using the formula, we study global sections of $\ntDsheaf{}$-modules on $G/U$.

Let $G$ be a simply-connected connected semisimple algebraic group over $\CC$.
Fix a Borel subgroup $B = TU$ of $G$ with a maximal torus $T$ and unipotent radical $U$.
Then $G/U$ is a quasi-affine $G$-variety and the natural projection $p\colon G/U\rightarrow G/B$ gives a principal $T$-bundle on the flag variety $G/B$.
The Beilinson--Bernstein correspondence \cite{BeBe81} says that the global section functor gives an equivalence of categories from the category $\Mod_{qc}(\ntDsheaf{G/B,\lambda})$ to that of $\lie{g}$-modules with infinitesimal character $\lambda \in \lie{t}^*$ if $\lambda$ is regular anti-dominant.
Here $\ntDsheaf{G/B, \lambda}$ is an algebra of twisted differential operators (TDO) on $G/B$.
See Subsection \ref{subsection:BB} for our parametrization.

As a global counterpart of the correspondence, a $\ntDsheaf{G/U}$-module is regarded as a good family of $\lie{g}$-modules.
In fact, for a quasi-coherent $\ntDsheaf{G/U}$-module $\Dmod{M}$, a family $\set{Dp_{+,\lambda}(\Dmod{M})}_{\lambda \in \lie{t}^*}$ of (complexes of) twisted $\ntDsheaf{}$-modules on $G/B$ is defined by using the direct image functors $Dp_{+,\lambda}$.
Our main purpose is to show fundamental results about $\ntDsheaf{G/U}$-modules and their global sections.
Using $\ntDsheaf{G/U}$-modules, we study the space $V/\lie{u}V$ of $\lie{u}$-coinvariants.
Set $\ntDalg{G/U}\coloneq \sect(\ntDsheaf{G/U})$.

Although $G/U$ is not affine, it is known that $\ntDalg{G/U}$ is a simple noetherian finitely generated $\CC$-algebra.
Moreover, the associated graded algebra $\gr \ntDalg{G/U}$ with respect to the order filtration
is isomorphic to $\rring{T^*(G/U)} = \sect(\rsheaf{T^*(G/U)})$, where $\rsheaf{T^*(G/U)}$ is the structure sheaf on $T^*(G/U)$.
These properties are proved in \cite[Theorem 1.1]{BBP02}, \cite[Proposition 1.1]{LeSt06} and \cite[Corollary 3.6.1]{GR15}.
This means that the algebra $\ntDalg{G/U}$ is similar to algebras of differential operators on affine spaces.
For example, holonomicity of $\ntDalg{G/U}$-modules can be defined as follows.

The definition of holonomic $\ntDalg{G/U}$-modules (Definition \ref{def:holonomic}) in this paper is standard.
We say that a $\ntDalg{G/U}$-module $M$ is holonomic if $M$ is finitely generated and the Gelfand--Kirillov dimension of $M$ is less than or equal to $\dim(G/U)$.
We show in Theorem \ref{thm:HolonomicLocalGlobal} that the global section functor $\sect$ and the localization functor $\loc$ preserve the holonomicity.

\begin{theorem}\label{intro:thm:holonomic}
	For any holonomic $\ntDsheaf{G/U}$-module $\Dmod{M}$, the global section $\sect(\Dmod{M})$ is also holonomic.
	Conversely, for any holonomic $\ntDalg{G/U}$-module $M$, the localization $\loc(M) = \ntDsheaf{G/U}\otimes_{\ntDalg{G/U}}M$ is holonomic.
\end{theorem}

The result for the localization is essentially proved in \cite[Lemma 3.2]{LeSt06}.
For the global section, we need to show that $\sect(\Dmod{M})$ is finitely generated.
To show this, we use the theory of multiplicities of filtered modules over filtered algebras.
See Subsection \ref{subsection:holonomicity}.

We shall state a generalization of the Beilinson--Bernstein correspondence.
Let $S$ be a multiplicative subset of $\univ{t}$ generated by
\begin{align*}
	\set{H_\alpha + \rho(H_\alpha) + i: \alpha \in \Delta^+, i \text{ is a non-negative integer}},
\end{align*}
where $\Delta^+$ is the set of all roots in $\lie{u}$ and $H_\alpha$ is the unique element of $\lie{t}$ satisfying $\mu(H_\alpha) = 2(\mu, \alpha)/(\alpha, \alpha)$ for any $\mu \in \lie{t}^*$.
Set $\ntDsheaf{G/B,S}\coloneq S^{-1}(p_* \ntDsheaf{G/U})^T$.
It is well-known (e.g.\ \cite[Proposition 8]{BoBr89}) that $\ntDalg{G/U}^T$ is isomorphic to $\univ{g}\otimes_{\univcent{g}}\univ{t}$,
where $\univcent{g}$ is the center of the universal enveloping algebra $\univ{g}$.
Hence we have $\sect(\ntDsheaf{G/B,S})\simeq \univ{g}\otimes_{\univcent{g}}S^{-1}\univ{t}$.
The following result is proved in Theorem \ref{thm:Daffine}.

\begin{theorem}\label{intro:thm:GenBB}
	The global section functor $\sect$ gives an equivalence of categories from the category $\Mod_{qc}(\ntDsheaf{G/B,S})$ to $\Mod(\univ{g}\otimes_{\univcent{g}} S^{-1}\univ{t})$.
\end{theorem}

$\ntDsheaf{G/B,\lambda}$-modules with regular anti-dominant $\lambda$ are also $\ntDsheaf{G/B,S}$-modules.
Hence Theorem \ref{intro:thm:GenBB} is a generalization of the Beilinson--Bernstein correspondence.
Remark that we use the Beilinson--Bernstein correspondence in our proof to show non-vanishing of $\sect(\Dmod{M})$ (Lemma \ref{lem:nonVanishing}).

A different generalization of the Beilinson--Bernstein correspondence is given by Ben-Zvi--Nadler \cite{BeNa19}.
They used a categorical approach.
Their result can be applied to all $\lie{g}$-modules without our localization $S^{-1}$.
It is not certain that Theorem \ref{intro:thm:GenBB} follows from their result.

Let $(\alg{A}, G)$ be a generalized pair
and $V$ an irreducible $\alg{A}$-module.
Another main object in this paper is the space $V/\lie{u}V$ of $\lie{u}$-coinvariants.
The $\lie{t}$-module $V/\lie{u}V$ is written as $V\otimes_{\univ{g}} (\univ{g}\otimes_{\univ{u}} \CC)$.
It is important that $\univ{g}\otimes_{\univ{u}}\CC$ (modulo some $\lie{t}$-character twist) has a natural $\ntDsheaf{G/U}$-module structure.
It is the unique irreducible holonomic $\ntDsheaf{G/U}$-module supported on $\set{eU}$.
Using this action, let $(\alg{A}\otimes \ntDalg{G/U})^G$ act on $V/\lie{u}V$.
We show that the $(\alg{A}\otimes \ntDalg{G/U})^G$-module has finite length under a strong assumption on $\alg{A}$.
A key point is that $V/\lie{u}V$ is represented by using the Zuckerman derived functor:
\begin{align*}
	V/\lie{u}V \simeq H^n(\lie{g}; V\otimes (\univ{g}\otimes_{\univ{u}} \CC))
	\simeq \Dzuck{G}{}{n}(V\otimes_{\univ{g}} (\univ{g}\otimes_{\univ{u}} \CC))^G,
\end{align*}
where $n = \dim(\lie{g})$.
We omit tensoring a $\lie{t}$-character.

Let $X$ be a smooth $G$-variety and $\Dsheaf{A}{X}$ a $G$-equivariant TDO on $X$.
Suppose that the global section functor $\sect$ is exact on $\Mod_{qc}(\Dsheaf{A}{X})$.
Set $\Dalg{A}{X} \coloneq \sect(\Dsheaf{A}{X})$.
It is known that $\sect(\Dmod{M})$ is irreducible or zero for irreducible $\Dmod{M} \in \Mod_{qc}(\Dsheaf{A}{X})$ (Fact \ref{fact:EquivalenceCategory}).
Using a localization of the Zuckerman derived functor (Proposition \ref{prop:LocalizationZuckerman}) and Theorem \ref{intro:thm:holonomic}, we obtain

\begin{theorem}\label{intro:thm:FiniteLength}
	Let $\Dmod{M}$ be an irreducible holonomic $\Dsheaf{A}{X}$-module and $i \in \NN$.
	Then the $(\Dalg{A}{X}\otimes \ntDalg{G/U})^G$-module $H^i(\lie{u}; \sect(\Dmod{M}))$ has finite length.
\end{theorem}

In \cite{Ki24-2}, we introduced a notion of small Cartan subalgebras for $\lie{g}$-modules.
For a non-zero $\lie{t}$-module $V$, we say that $V$ has a small Cartan subalgebra $\lie{a} = \lie{t}/\lie{m}$
if any $\Variety(P)$ ($P \in \Ass_{\univ{t}}(V)$) is a translation of $\lie{a}^*$,
where $\Variety(P)$ is the closed subvariety of $\lie{t}^*$ determined by the prime ideal $P$.
We have shown in \cite[Theorem 3.25]{Ki24-2} that any irreducible $\alg{A}$-module has a small Cartan subalgebra if $\alg{A}$ has at most countable dimension.
We will give another proof of this result in Theorem \ref{thm:Affinity}.
The following theorem is an application of Theorem \ref{intro:thm:FiniteLength}.

\begin{theorem}\label{intro:thm:Tsupport}
	Let $\Dmod{M}$ be an irreducible holonomic $\Dsheaf{A}{X}$-module and $i \in \NN$.
	Then $H^i(\lie{u}; \sect(\Dmod{M}))$ has a $\lie{t}$-module finite exhaustive filtration $0 = V_0\subset V_1 \subset \cdots \subset V_r$ such that each $V_j / V_{j-1}$ has a small Cartan subalgebra.
\end{theorem}

We prove Theorems \ref{intro:thm:FiniteLength} and \ref{intro:thm:Tsupport} in Subsection \ref{subsection:LargeAction}.
Roughly speaking, Theorem \ref{intro:thm:Tsupport} says that the support of a vector $v \in V/\lie{u}V$ is a finite union of affine subspaces in $\lie{t}^*$.
We rephrase the results for the branching problem of reductive Lie groups in Subsection \ref{subsection:Branching}.

This paper is organized as follows.
In Section 2, we recall fundamental results about the Zuckerman derived functor and twisted $\ntDsheaf{}$-modules.
In Section 3, we deal with the basic affine space.
We recall the formula by Bezrukavnikov--Braverman--Positselskii, and show Theorem \ref{intro:thm:holonomic}.
Section 4 is devoted to the generalization of the Beilinson--Bernstein correspondence.
We prove Theorem \ref{intro:thm:GenBB} here.
Theorems \ref{intro:thm:FiniteLength} and \ref{intro:thm:Tsupport} are proved in Section 5.

\subsection*{Notation and convention}

In this paper, any algebra except Lie algebras is unital, associative and over $\CC$,
and any variety is quasi-projective and defined over $\CC$.
Let $X$ be a smooth variety.
We denote by $\rsheaf{X}$ and $\rring{X}$ the structure sheaf on $X$ and its global section,
and denote by $\ntDsheaf{X}$ and $\ntDalg{X}$ the sheaf of differential operators on $X$ and its global section.
We express TDOs and the spaces of their global sections by script letters and corresponding calligraphic letters, respectively.
For example, the spaces of global sections of TDOs $\Dsheaf{A}{X}$ and $\Dsheaf{A}{X,\lambda}$ are denoted as $\Dalg{A}{X}$ and $\Dalg{A}{X,\lambda}$, respectively.
For an ideal $I$ of $\rring{X}$, we denote by $\Variety(I)$ the subvariety in $X$ determined by $I$.

We express algebraic groups and their Lie algebras by Roman alphabets and corresponding German letters, respectively.
For example, the Lie algebras of algebraic groups $G$, $K$ and $H$ are denoted as $\lie{g}$, $\lie{k}$ and $\lie{h}$, respectively.
For an algebraic group $G$, let $G_0$ denote the identity component of $G$.

For a $G$-set $X$ of a group $G$, we write $X^G$ for the subset of all $G$-invariant elements in $X$.
We use similar notations for the set of all vectors annihilated by a Lie algebra $\lie{g}$ (resp.\ an ideal $I$) as $V^{\lie{g}}$ (resp.\ $V^I$).
We denote by $\CC_{\lambda}$ the representation space of a character $\lambda$ of a Lie algebra $\lie{t}$.
When $\lie{t}$ is commutative, there are two ways to regard $\CC_{\lambda}$ as a right $\univ{t}$-module.
In the case, let $\univ{t}$ act on $\CC_{\lambda}$ via $v\cdot X = \lambda({}^tX)v$ ($X \in \univ{t}$, $v \in \CC_{\lambda}$) to satisfy $\CC_{\lambda}\otimes_{\univ{t}}V \simeq (\CC_{\lambda}\otimes V)/\lie{t}(\CC_{\lambda}\otimes V)$
for any $\lie{t}$-module $V$.

We denote by $\Mod(\cdot)$ (resp.\ $\Mod_{qc}(\cdot)$, $\Mod_{h}(\cdot)$) the category of left modules (resp.\ quasi-coherent modules, holonomic modules) of an algebra (or a sheaf of algebras).
We use the standard notation $D^b(\cdot)$, $D^b_{qc}(\cdot)$ and $D^b_h(\cdot)$ for the corresponding derived categories.
For example, $D^b_h(\alg{A})$ is the full subcategory of $D^b(\alg{A})$ with holonomic cohomologies.
We express the exterior tensor product of sheaves by $\boxtimes$.
See \cite[p.38]{HTT08} for the exterior tensor product.

\subsection*{acknowledgement}

This work was supported by JSPS KAKENHI Grant Number JP23K12963.

\section{Preliminary}

In this section, we review basic results about $\ntDsheaf{}$-modules and (a localization of) the Zuckerman derived functor.

\subsection{Zuckerman derived functor}

We review the Zuckerman derived functor.
We refer the reader to \cite[Chapter 6]{Wa88_real_reductive_I} and \cite{KnVo95_cohomological_induction}.

Let $G$ be a connected reductive algebraic group over $\CC$.
In this paper, a $G$-module means a locally finite representation of the group $G$ whose finite-dimensional subrepresentations are representations as an algebraic group.
We say that the $G$-action on a $G$-module is rational.
Since $G$ is reductive, a $G$-module is completely reducible.
We shall recall the definition of generalized pairs.
See \cite[p.96]{KnVo95_cohomological_induction}.

\begin{definition}
	Let $\alg{A}$ be a (unital associative) $\CC$-algebra equipped with a $G$-module structure and a $G$-equivariant homomorphism $\iota\colon \univ{g}\rightarrow \alg{A}$.
	We say that $(\alg{A}, G)$ is a \define{generalized pair} if the $G$-action on $\alg{A}$ is given by automorphisms
	and the adjoint $\lie{g}$-action on $\alg{A}$ coincides with the differential action of $G$.
\end{definition}

\begin{definition}
	Let $(\alg{A}, G)$ be a generalized pair and $V$ an $\alg{A}$-module.
	We say that $V$ is an \define{$(\alg{A}, G)$-module} if the $\lie{g}$-action on $V$ lifts to a rational $G$-action.
	The category $\Mod(\alg{A}, G)$ of $(\alg{A}, G)$-modules is defined in a natural way.
\end{definition}

Since $G$ is assumed to be connected, our definition is easier than the general one in \cite{KnVo95_cohomological_induction}.
It is easy to see that any $\alg{A}$-homomorphism between $(\alg{A}, G)$-modules is $G$-equivariant.
In other words, the category $\Mod(\alg{A}, G)$ is a full subcategory of $\Mod(\alg{A})$.

Let $(\alg{A}, G)$ be a generalized pair and $V$ an $\alg{A}$-module.
Let $i \in \NN$ and set
\begin{align*}
	\Dzuck{G}{}{i}(V) \coloneq H^i(\lie{g}; \rring{G}\otimes V),
\end{align*}
where the $\lie{g}$-cohomology is taken via the tensor product of the right regular action on $\rring{G}$ and the action on $V$.
Then the left regular action on $\rring{G}$ induces a rational $G$-action on $\Dzuck{G}{}{i}(V)$.
Moreover, it is known that $\Dzuck{G}{}{i}(V)$ admits an $(\alg{A}, G)$-module structure and satisfies the following property.
See \cite[Proposition I.8.2]{BoWa00_continuous_cohomology}.
The functor $\Dzuck{G}{}{i}$ is called the $i$-th Zuckerman derived functor.

\begin{fact}\label{fact:GInvZuckerman}
	There exists a natural $\alg{A}^G$-module isomorphism $\Dzuck{G}{}{i}(V)^G \simeq H^i(\lie{g}; V)$.
	In particular, the length of $H^i(\lie{g}; V)$ is bounded by that of $\Dzuck{G}{}{i}(V)$ from above.
\end{fact}

Note that the functor $(\cdot)^G \colon \Mod(\alg{A}, G)\rightarrow \Mod(\alg{A}^G)$ is exact and sends irreducible objects to zero or irreducible objects since $G$ is reductive.
The second assertion in Fact \ref{fact:GInvZuckerman} follows from this.

\subsection{Twisted \texorpdfstring{$\ntDsheaf{}$}{D}-module}

We recall the notion of twisted $\ntDsheaf{}$-modules.
We refer the reader to \cite{HTT08}.
See also \cite[1.1]{KaTa96} and \cite[A.1]{HMSW87}.
In this paper, any algebraic variety is quasi-projective and defined over $\CC$.

Let $X$ be a smooth variety.
Let $\rsheaf{X}$ denote the structure sheaf on $X$ and $\ntDsheaf{X}$ denote the sheaf of differential operators on $X$.
Then there is a canonical monomorphism $\iota_X \colon \rsheaf{X}\hookrightarrow \ntDsheaf{X}$, and $\ntDsheaf{X}$ is a quasi-coherent left/right $\rsheaf{X}$-module.

\begin{definition}
	Let $\Dsheaf{A}{X}$ be a sheaf of algebras on $X$ equipped with a monomorphism $\iota\colon \rsheaf{X} \hookrightarrow \Dsheaf{A}{}$.
	We say that $\Dsheaf{A}{X}$ is an algebra of twisted differential operators (or \define{TDO}) on $X$ if
	there exists an open covering $X = \bigcup_i U_i$ such that $(\Dsheaf{A}{X}|_{U_i}, \iota) \simeq (\ntDsheaf{U_i}, \iota_{U_i})$ for any $i$.
\end{definition}

Let $\Dsheaf{A}{X}$ be a TDO on $X$.
An $\Dsheaf{A}{X}$-module $\Dmod{M}$ is said to be \define{quasi-coherent} if $\Dmod{M}$ is quasi-coherent as an $\rsheaf{X}$-module.
Let $\Mod_{qc}(\Dsheaf{A}{X})$ denote the category of quasi-coherent $\Dsheaf{A}{X}$-modules, $D^b(\Dsheaf{A}{X})$ the bounded derived category of $\Dsheaf{A}{X}$-modules,
and $D^b_{qc}(\Dsheaf{A}{X})$ the full subcategory of $D^b(\Dsheaf{A}{X})$ consisting of complexes with quasi-coherent cohomologies.

By the existence of local trivializations, $\Dsheaf{A}{X}$ has a canonical filtration called the order filtration induced from that of $\ntDsheaf{U_i}$,
and the associated graded algebra $\gr(\Dsheaf{A}{X})$ is isomorphic to $p_*(\rsheaf{T^*X})$.
Here $p\colon T^* X\rightarrow X$ is the natural projection from the cotangent bundle of $X$ to $X$.
A quasi-coherent $\Dsheaf{A}{X}$-module $\Dmod{M}$ is said to be \define{coherent} if $\Dmod{M}$ is locally finitely generated over $\Dsheaf{A}{X}$.
Let $\Dmod{M}$ be a coherent $\Dsheaf{A}{X}$-module.
Then $\Dmod{M}$ has a good filtration.
The \define{characteristic variety} $\Ch(\Dmod{M}) \subset T^*X$ is defined by the support of the $\rsheaf{T^*X}$-module $\rsheaf{T^*X}\otimes_{p^{-1}p^*(\rsheaf{T^*X})}p^{-1}(\gr(\Dmod{M}))$.

It is well-known that $\dim(\Ch(\Dmod{M}))\geq \dim(X)$ if $\Dmod{M}\neq 0$ (see \cite[Corollary 2.3.2]{HTT08}).
We say that a coherent $\Dsheaf{A}{X}$-module $\Dmod{M}$ is \define{holonomic} if $\dim(\Ch(\Dmod{M}))\leq \dim(X)$.
It is easy to see that any holonomic $\Dsheaf{A}{X}$-module has finite length (see \cite[Proposition 3.1.2]{HTT08}).
Let $\Mod_{h}(\Dsheaf{A}{X})$ denote the category of holonomic $\Dsheaf{A}{X}$-modules
and $D^b_{h}(\Dsheaf{A}{X})$ the full subcategory of $D^b(\Dsheaf{A}{X})$ consisting of complexes with holonomic cohomologies.

Recall the direct image functor and the inverse image functor.
Let $f\colon X\rightarrow Y$ be a morphism between smooth varieties
and $\Dsheaf{A}{Y}$ a TDO on $Y$.
We set
\begin{align*}
\Omega_{f}&\coloneq f^{-1}\Omega_{Y}^{\invshf}\otimes_{f^{-1}\rsheaf{Y}}\Omega_{X},\\
\Dsheaf{A}{Y\leftarrow X} &\coloneq f^{-1}\Dsheaf{A}{Y}\otimes_{f^{-1}\rsheaf{Y}}\Omega_{f},\\
\Dsheaf{A}{X\rightarrow Y} &\coloneq \rsheaf{X} \otimes_{f^{-1}\rsheaf{Y}}f^{-1}\Dsheaf{A}{Y},
\end{align*}
where $\Omega_{X}$ (resp.\ $\Omega_{Y}$) denotes the canonical sheaf of $X$ (resp.\ $Y$).
\begin{definition}
	Let $f^{\#}\Dsheaf{A}{Y}$ denote the sheaf of all differential endomorphisms
	of the $\rsheaf{X}$-module $\Dsheaf{A}{X\rightarrow Y}$ that commute with the right $f^{-1}\Dsheaf{A}{Y}$-action.
\end{definition}
Then $f^{\#}\Dsheaf{A}{Y}$ is a TDO on $X$
and $\Dsheaf{A}{Y\leftarrow X}$ is an $(f^{-1}\Dsheaf{A}{Y}, f^{\#}\Dsheaf{A}{Y})$-bimodule.
Note that $f^{\#}\ntDsheaf{Y}$ is canonically isomorphic to $\ntDsheaf{X}$.
See \cite[Definition 1.3.1]{HTT08}.
The direct image of $\cpx{\Dmod{M}} \in D^b_{qc}(f^{\#}\Dsheaf{A}{Y})$ is defined by
\begin{align*}
D f_{+}(\cpx{\Dmod{M}}) = Rf_*(\Dsheaf{A}{Y\leftarrow X}\otimes^L_{f^{\#}\Dsheaf{A}{Y}}\cpx{\Dmod{M}}) \in D^b_{qc}(\Dsheaf{A}{Y}),
\end{align*}
and the inverse image of $\cpx{\Dmod{N}} \in D^b_{qc}(\Dsheaf{A}{Y})$ is defined by
\begin{align*}
	Lf^*(\cpx{\Dmod{N}}) = \Dsheaf{A}{X\rightarrow Y}\otimes^L_{f^{-1}\Dsheaf{A}{Y}}f^{-1}(\cpx{\Dmod{N}})
	\in D^b_{qc}(f^\#\Dsheaf{A}{Y}).
\end{align*}
It is well-known that the functors $D f_{+}$ and $Lf^*$ are local for $Y$.
Since the functors preserve the holonomicity, we have the two functors
\begin{align*}
	D f_{+}\colon D^b_{h}(\Dsheaf{A}{Y}) \rightarrow D^b_{h}(f^{\#}\Dsheaf{A}{Y}), \quad 
	L f^* \colon D^b_{h}(f^{\#}\Dsheaf{A}{Y}) \rightarrow  D^b_{h}(\Dsheaf{A}{Y}).
\end{align*}
See \cite[Theorem 3.2.3]{HTT08}.

We shall review the notions of $G$-equivariant TODs and modules.
Let $G$ be a connected reductive algebraic group and $X$ a smooth $G$-variety.
Write $p, m\colon G\times X\rightarrow X$ for the projection and the multiplication map, respectively.
An $\rsheaf{X}$-module $\Dmod{M}$ is said to be \define{$G$-equivariant} if an $\rsheaf{G\times X}$-module isomorphism $m^* \Dmod{M} \xrightarrow{\simeq} p^* \Dmod{M}$ is specified and satisfies the associative law \cite[(3.1.2)]{Ka08_dmodule}.
We call the $G$-equivariant structure an (algebraic) $G$-action on $\Dmod{M}$.
In fact, $G$ acts on the set of local sections of $\Dmod{M}$.
By the algebraicity, the action is differentiable, that is, it induces a Lie algebra homomorphism $\lie{g}\rightarrow \End_{\CC_X}(\Dmod{M})$.

\begin{definition}
	We say that a TDO $\Dsheaf{A}{X}$ on $X$ is \define{$G$-equivariant} if a $G$-action on $\Dsheaf{A}{X}$ and a $G$-equivariant algebra homomorphism $i_{\lie{g}}\colon \univ{g}\rightarrow \Dsheaf{A}{X}$ are specified and satisfy the following conditions.
	\begin{enumerate}
		\item The $G$-action is given by algebra isomorphisms.
		\item The differential of the $G$-action on $\Dsheaf{A}{X}$ coincides with the adjoint action of $\lie{g}$ on $\Dsheaf{A}{X}$ given by $i_{\lie{g}}$.
	\end{enumerate}
	An $\Dsheaf{A}{X}$-module $\Dmod{M}$ is said to be \define{$G$-equivariant} or an \define{$(\Dsheaf{A}{X}, G)$-module} if $\Dmod{M}$ is $G$-equivariant as an $\rsheaf{X}$-module and the isomorphism $m^* \Dmod{M} \xrightarrow{\simeq} p^* \Dmod{M}$ is a $(\ntDsheaf{G}\boxtimes \Dsheaf{A}{X})$-isomorphism.
\end{definition}

It is easy to see that $(\Dalg{A}{X}, G)$ forms a generalized pair and $\sect(\Dmod{M})$ is an $(\Dalg{A}{X}, G)$-module.
Note that the $G$-equivariant structure on $\Dsheaf{A}{X}$ induces an isomorphism $m^\# \Dsheaf{A}{X} \simeq p^\# \Dsheaf{A}{X}(\simeq \ntDsheaf{G}\boxtimes \Dsheaf{A}{X})$ satisfying the associative law.
See \cite[Lemma 1.8.7]{BeBe93}.

In this paper, we use the functors only for easy $f$, projections of fiber bundles and principal bundles.
Let $X$ and $Y$ be smooth varieties and $f \colon X\times Y\rightarrow Y$ the projection onto the second factor.
Let $\Dsheaf{A}{Y}$ be a TDO on $Y$.
Then we have $f^\# \Dsheaf{A}{Y} \simeq \ntDsheaf{X}\boxtimes \Dsheaf{A}{Y}$.
The following proposition is easy from the definition.

\begin{proposition}\label{prop:InverseImageDirectProduct}
	The inverse image functor $f^*$ is exact and
	$f^*(\Dmod{M}) \simeq \rsheaf{X}\boxtimes \Dmod{M}$ holds for any $\Dmod{M} \in \Mod_{qc}(\Dsheaf{A}{Y})$.
\end{proposition}

Recall a result about the direct image functor with respect to the projections of principal bundles.
Let $G$ be a connected reductive algebraic group and $f \colon X\rightarrow Y$ a principal $G$-bundle on a smooth variety $Y$.
In this paper, any principal bundle is locally trivial in the Zariski topology.
Let $\Dsheaf{A}{X}$ be a $G$-equivariant TDO on $X$.
For each $\lambda \in \Hom_{G}(\lie{g}^*, \CC)$, we set
\begin{align*}
	\Dsheaf{A}{Y,\lambda} \coloneq (f_*(\Dsheaf{A}{X})\otimes_{\univ{g}} \End_{\CC}(\CC_{\lambda}))^G.
\end{align*}
Then $\Dsheaf{A}{Y,\lambda}$ is a TDO on $Y$, and if $\lambda$ lifts to a character of $G$, $\Dsheaf{A}{Y,\lambda}$ acts on $(f_*(\rsheaf{X})\otimes \CC_{\lambda})^G$, which is the sheaf of local sections of the line bundle $X\times_G \CC_{\lambda}$.
Moreover, $f^\# \Dsheaf{A}{Y, \lambda}$ is naturally isomorphic to $\Dsheaf{A}{X}$ (see \cite[Lemma 3.11]{Ki23}).
In this case, the direct image functor $Df_{+,\lambda}$ can be represented by a simple way
as we have seen in \cite[Proposition 3.14]{Ki23}.

\begin{fact}\label{fact:DirectImagePrincipal}
	Let $\lambda \in \Hom_{G}(\lie{g}^*, \CC)$.
	Then the functor $Df_{+,\lambda}$ is isomorphic to the left derived functor of the functor
	$\Mod_{qc}(\Dsheaf{A}{X}) \ni \Dmod{M} \mapsto \CC_{\lambda} \otimes_{\univ{g}} f_*(\Dmod{M}) \in \Mod_{qc}(\Dsheaf{A}{Y,\lambda})$.
\end{fact}

Note that $f_*$ is an exact functor on $\Mod_{qc}(\Dsheaf{A}{X})$ since $f$ is affine.
To compute $Df_{+,\lambda}(\Dmod{M})$, we need to take a locally projective resolution of $\Dmod{M}$.
It is easy to see that $\sect(V, f_*\Dmod{P})$ is a projective $\univ{g}$-module for any locally projective $\Dmod{P} \in \Mod_{qc}(\Dsheaf{A}{X})$ and affine open subset $V\subset Y$.
See \cite[Lemma 3.15]{Ki23}.

\subsection{Localization of Zuckerman derived functor}

We shall consider a localization of the Zuckerman derived functor.
We refer the reader to \cite{Bi90}, \cite{Ki12} and \cite[Section 6]{Ki23}.

Let $X$ be a smooth $G$-variety and $\Dsheaf{A}{X}$ a $G$-equivariant TDO on $X$.
Write $p, m\colon G\times X\rightarrow X$ for the projection and the multiplication map.
For $\Dmod{M} \in \Mod_{qc}(\Dsheaf{A}{X})$ and $i \in \NN$, we set
\begin{align*}
	\DlocZuck{G}{}{i}(\Dmod{M})\coloneq H^{-n+i}\circ Dm_{+}\circ p^*(\Dmod{M}) \in \Mod(\Dalg{A}{X}),
\end{align*}
where $n = \dim(\lie{g})$.
By Facts \ref{prop:InverseImageDirectProduct} and \ref{fact:DirectImagePrincipal}, we have $\DlocZuck{G}{}{i}(\Dmod{M}) \simeq H^i(\lie{g}; m_*(\rsheaf{G}\boxtimes \Dmod{M}))$.
Note that $H^i(\lie{g};\cdot )\simeq H_{n-i}(\lie{g}; \cdot)$ by the Poincar\'e duality (see \cite[Corollary 3.6]{HTT08}).
The functor $\DlocZuck{G}{}{i}$ is a localization of the Zuckerman derived functor $\Dzuck{G}{}{i}$ in the following sense.

\begin{proposition}\label{prop:LocalizationZuckerman}
	Let $\Dmod{M} \in \Mod_{qc}(\Dsheaf{A}{X})$ be an acyclic sheaf.
	Take a free resolution $\cpx{F}$ of the $\lie{g}$-module $\CC$.
	Then there exists an isomorphism 
	\begin{align*}
		R\sect(Dm_{+}\circ p^*(\Dmod{M})) \simeq \cpx{F} \otimes_{\univ{g}} (\rring{G}\otimes \sect(\Dmod{M}))
	\end{align*}
	in $D^b(\Dalg{A}{X})$.
\end{proposition}

\begin{proof}
	Since $\Dmod{M}$ is acyclic, $\cpx{F}\otimes_{\univ{g}} m_*(\rsheaf{G}\boxtimes \Dmod{M})$ is a complex of acyclic sheaves,
	and hence
	\begin{align*}
		R\sect(Dm_{+}\circ p^*(\Dmod{M})) &=  R\sect(\cpx{F}\otimes_{\univ{g}} m_*(\rsheaf{G}\boxtimes \Dmod{M})) \\
		&\simeq \cpx{F} \otimes_{\univ{g}} \sect(m_*(\rsheaf{G}\boxtimes \Dmod{M})) \\
		&\simeq \cpx{F} \otimes_{\univ{g}} (\rring{G}\otimes \sect(\Dmod{M}))
	\end{align*}
	by Fact \ref{fact:DirectImagePrincipal}.
	We have shown the assertion.
\end{proof}

Taking the cohomology of $\cpx{F} \otimes_{\univ{g}} (\rring{G}\otimes \sect(\Dmod{M}))$, we obtain the $G$-module $\Dzuck{G}{}{i}(\sect(\Dmod{M}))$.
The $\Dalg{A}{X}$-action on $\Dzuck{G}{}{i}(\sect(\Dmod{M}))$ to define the Zuckerman derived functor coincides with the one given by the $\Dsheaf{A}{X}$-action on $R\sect(Dm_{+}\circ p^*(\Dmod{M}))$.
See \cite[I.3.1]{BoWa00_continuous_cohomology} for the $\Dalg{A}{X}$-action on $\rring{G}\otimes \sect(\Dmod{M})$.
We omit the description of the action.

\begin{corollary}\label{cor:FiniteZuckerman}
	Let $\Dmod{M} \in \Mod_{qc}(\Dsheaf{A}{X})$ be an acyclic sheaf.
	Assume that all $H^p(X, \DlocZuck{G}{}{q}(\Dmod{M}))$ have finite length.
	Then all $\Dzuck{G}{}{i}(\sect(\Dmod{M}))$ have finite length.
\end{corollary}

\begin{proof}
	Take a free resolution $\cpx{F}$ of the $\lie{g}$-module $\CC$.
	Then we have isomorphisms
	\begin{align*}
		\Dzuck{G}{}{i}(\sect(\Dmod{M})) \simeq H^{i-n}(\cpx{F}\otimes_{\univ{g}}(\rring{G}\otimes \sect(\Dmod{M}))) \\
		\simeq H^{i-n}\circ R\sect(Dm_{+}\circ p^*(\rsheaf{G}\boxtimes \Dmod{M}))
	\end{align*}
	of $\Dalg{A}{X}$-modules by Proposition \ref{prop:LocalizationZuckerman}.
	By using the truncation functor for $Dm_{+}\circ p^*(\rsheaf{G}\boxtimes \Dmod{M})$ and the long exact sequence,
	the assertion is proved.
	The argument is standard, so we omit the details.
	See e.g.\ \cite[Subsection 2.3]{Ki23}.
\end{proof}

\section{Basic affine space}

In this section, we deal with $\ntDsheaf{}$-modules on the basic affine space 
and their global sections.
We can consider a $\ntDsheaf{}$-module on the basic affine space $G/U$
as a `good' family of $\lie{g}$-modules.
We recall the key fact by Bezrukavnikov--Braverman--Positselskii \cite{BBP02}
and show fundamental results about $\ntDsheaf{}$-modules on $G/U$.

\subsection{Differential operators on basic affine space}

Let $G$ be a simply-connected connected semisimple algebraic group.
Fix a Borel subgroup $B=TU$ of $G$ with a maximal torus $T$ and unipotent radical $U$.
Let $\roots$ denote the set of roots for $(\lie{g}, \lie{t})$
and $\proots$ the set of all roots in $\lie{u}$.
Set $\rho\coloneq \frac{1}{2}\sum_{\alpha \in \proots} \alpha$.
Write $W_G$ for the Weyl group of $G$.

The variety $G/U$ is called the basic affine space.
The basic affine space $G/U$ is a quasi-affine variety with the natural principal $T$-bundle structure $p\colon G/U \rightarrow G/B$.
The coordinate ring $\rring{G/U}$ has the irreducible decomposition
\begin{align*}
	\rring{G/U} = \bigoplus_{\lambda} \rring{G/U}(\lambda)
\end{align*}
as a $G$-module, where the sum is taken over all dominant integral weights $\lambda$
and $\rring{G/U}(\lambda)$ denotes the isotypic component with the highest weight $\lambda$.
Then each $\rring{G/U}(\lambda)$ is irreducible and $\rring{G/U}(\lambda)\cdot \rring{G/U}(\mu) = \rring{G/U}(\lambda + \mu)$
for any $\lambda$ and $\mu$.
In particular, $\rring{G/U}$ is finitely generated.

Since $\rring{G/U}$ is finitely generated, $G/U$ has a canonical affine closure $\overline{G/U}$,
that is, $\overline{G/U}$ is the affine variety with the coordinate ring $\rring{G/U}$.
It is known that $\overline{G/U}$ is normal and $\overline{G/U}-G/U$ has codimension greater than 2.
In fact, $G$-stable radical ideals of $\rring{G/U}$ can be easily described by fundamental weights.
For a $\ntDalg{G/U}$-module $M$, its localization $\loc(M)$ is defined by $\loc(M)\coloneq \ntDsheaf{G/U}\otimes_{\ntDalg{G/U}}M$.
By definition, $\loc$ is left adjoint to the global section functor $\sect$.
The following properties are basic.
See e.g.\ \cite[Subsection 1.1]{BBP02} and \cite[Section 2]{LeSt06} for the details.

\begin{proposition}\label{prop:FundamentalBasicAffine}
	Let $M$ be a $\ntDalg{G/U}$-module.
	\begin{enumerate}
		\item $\loc(M)$ is naturally isomorphic to $\rsheaf{G/U}\otimes_{\rring{G/U}}M$ as an $\rsheaf{G/U}$-module.
		\item The localization functor $\loc$ is exact.
		\item The counit $\loc\circ \sect \rightarrow \id$ is a natural isomorphism.
	\end{enumerate}
\end{proposition}

Note that $\overline{G/U}$ may be singular in general.
This is the reason why the relation between $\Mod_{qc}(\ntDsheaf{G/U})$ and $\Mod(\ntDalg{G/U})$
is more complicated than that of $G/B$.

We state three facts related to the algebraic structure of $\ntDalg{G/U}$.
\begin{fact}\label{fact:dalg}
	Under the above setting, the algebra $\ntDalg{G/U}$ satisfies the following properties.
	\begin{enumerate}
		\item(\cite[Theorem 1.1]{BBP02}) $\ntDalg{G/U}$ is left and right Noetherian.
		\item(\cite[Theorem 1.1]{LS06}) $\ntDalg{G/U}$ is a finitely generated simple ring satisfying the 
			Auslander--Gorenstein and Cohen--Macaulay conditions.
		\item(\cite[Corollary 3.6.1, Lemma 3.6.2]{GR15}) The associated graded ring $\gr \ntDalg{G/U}$ with respect to the order filtration is isomorphic to the finitely generated $\CC$-algebra $\rring{T^*(G/U)}$.
	\end{enumerate}
\end{fact}

We denote by $L$ the left translation of $G$ on $\rring{G/U}$
and by $R$ the right translation of $T$ on $\rring{G/U}$.
Then the action is given by
\begin{align}
	L(g)R(t)f(\cdot)= f(g^{-1}\cdot t) \quad (g \in G, t \in T, f \in \rring{G/U}). \label{eqn:ActionGT}
\end{align}
For an integral weight $\lambda \in \lie{t}^*$, the $\rsheaf{G/B}$-module $(p_* \rsheaf{G/U} \otimes \CC_{\lambda})^T$
is isomorphic to the invertible sheaf $\Dmod{L}_{\lambda}$ of local sections of the line bundle $G\times_{B}\CC_{\lambda} \rightarrow G/B$.
Let $\rring{G/U}^\mu$ denote the weight space of the $T$-module $(R, \rring{G/U})$ with weight $\mu$.
In our convention, we have
\begin{align*}
	\rring{G/U}(\lambda) = \rring{G/U}^{-\longw(\lambda)} \simeq \sect(G/B, \Dmod{L}_{\longw(\lambda)}),
\end{align*}
where $\longw$ is the longest element of the Weyl group $W_G$ of $G$.
The image of $\univcent{g}$ by the composition $\univcent{g}\xrightarrow{L} L(\univ{g})\cap R(\univ{t}) \xrightarrow{R^{-1}} \univ{t}$ is the subalgebra of $W$-invariants via the action
\begin{align}
	w\cdot H = w(H) - \rho(H) + \rho(w(H)) \quad (w \in W_G, H \in \lie{t}). \label{eqn:Waction}
\end{align}
Remark that the $T$-action on $\rring{G/U}$ in \cite{BBP02} is a little different from our action.

\subsection{Review of Fourier transforms on basic affine space}
\label{subsection:reviewFourier}

Here we review Fourier transforms on the basic affine space and the results of Bezrukavnikov--Braverman--Positselskii \cite{BBP02}.
We refer the reader to \cite[Section 1]{Ka95} and \cite[Section 3]{BBP02}.
Retain the notation in the previous subsection.

First we recall the symplectic Fourier transform of differential operators on $\CC^2$.
Write $(x, y)$ for the standard coordinate of $\CC^2$.
The symplectic Fourier transform $F$ of $\ntDalg{\CC^2}$ is the automorphism satisfying
\begin{align*}
	F(x) = -\frac{\partial}{\partial y}, \quad F(y) = \frac{\partial}{\partial x}, \quad 
	F\left(\frac{\partial}{\partial x}\right) = y, \quad F\left(\frac{\partial}{\partial y}\right) = -x,
\end{align*}
which is induced from the symplectic form $(\cdot, \cdot)$ on $\CC^2$ with $(e_1, e_2) = 1$.
It is easy to see that $F$ is $\LieSL(2,\CC)$-equivariant and an involution.
If the symplectic form $(\cdot, \cdot)$ is replaced with its scalar multiple, then $F$ is twisted by the $\CC^\times$-action on $\ntDalg{\CC^2}$.

Return to the setting of $G/U$.
Let $\alpha$ be a simple root in $\proots$.
We denote by $\LieSL_{\alpha}$ the subgroup of $G$ isomorphic to $\LieSL(2,\CC)$
whose Lie algebra is generated by $\lie{g}_{\alpha}\oplus \lie{g}_{-\alpha}$.
Write $P_{\alpha}$ for the parabolic subgroup of $G$ generated by $B$ and $\LieSL_{\alpha}$,
and $U_{\alpha}$ for the unipotent radical of $P_{\alpha}$.
Then we have $[P_{\alpha}, P_{\alpha}]=\LieSL_{\alpha}\cdot U_{\alpha} \supset U$.

The natural surjection $G/U \rightarrow G/[P_{\alpha}, P_{\alpha}]$ is a fiber bundle with fiber $[P_\alpha, P_\alpha]/U \simeq \LieSL_\alpha / (\LieSL_\alpha \cap U) \simeq \CC^2-\set{0}$.
The fiber bundle is embedded in a $G$-equivariant vector bundle $\overline{G/U}^{\alpha}$ with fiber $\CC^2$ via
\begin{align*}
	\xymatrix{
	G/U \simeq G\times_{[P_{\alpha},P_{\alpha}]} (\CC^2-\set{0}) \ar[r]\ar@{_(->}[d]& G/[P_{\alpha}, P_{\alpha}] \\
	\overline{G/U}^{\alpha} \coloneq G\times_{[P_{\alpha},P_{\alpha}]} \CC^2. \ar[ur] & 
	}
\end{align*}
The representation of $\LieSL_{\alpha} \subset [P_{\alpha},P_{\alpha}]$ on $\CC^2$ is isomorphic to the natural representation of $\LieSL(2,\CC)$,
and $U_{\alpha}$ acts on $\CC^2$ trivially.

Since $\CC^2$ has a unique $\LieSL_{\alpha}$-invariant symplectic form up to scalar,
the vector bundle $\overline{G/U}^{\alpha}$ has a unique $G$-invariant symplectic form up to scalar via
\begin{align*}
	\overline{G/U}^{\alpha} \times_{G/[P_{\alpha}, P_{\alpha}]} \overline{G/U}^{\alpha}
	\simeq G\times_{[P_{\alpha}, P_{\alpha}]}(\CC^2\times \CC^2) \rightarrow \CC.
\end{align*}
Considering the symplectic Fourier transform fiberwise, we obtain an algebra automorphism of $\ntDalg{\overline{G/U}^{\alpha}}$.
Since the codimension of $\overline{G/U}^{\alpha}-G/U$ is greater than 2, we have $\ntDalg{\overline{G/U}^{\alpha}}=\ntDalg{G/U}$
and hence we obtain an automorphism $F_{s_{\alpha}}$ of $\ntDalg{G/U}$.
We call $F_{s_{\alpha}}$ a (partial) Fourier transform of $\ntDalg{G/U}$ corresponding to the simple reflection $s_{\alpha}$.
It is easy to see that $F_{s_{\alpha}}$ is $G$-equivariant because $F_{s_{\alpha}}$ is constructed from the $G$-invariant symplectic form.

For each simple root $\alpha$, we fix a $G$-invariant symplectic form on the vector bundle $\overline{G/U}^{\alpha}$
and the corresponding Fourier transform $F_{s_{\alpha}}$.
It is stated in \cite[Proposition 3.1]{BBP02} that the following fact is an unpublished work of S.\ Gelfand and M.\ Graev.

\begin{fact}\label{fact:WeylToAut}
	The assignment $s_{\alpha} \mapsto F_{s_{\alpha}} \in \Aut(\ntDalg{G/U})$ is extended uniquely to a group homomorphism $F_{\cdot}: W_G \rightarrow \Aut(\ntDalg{G/U})$ of the Weyl group $W_G$ of $G$.
\end{fact}

\begin{remark}
	The $T$-action on $\ntDalg{G/U}$ commutes with the $G$-action.
	It is not hard to see that $\Aut(\ntDalg{G/U})^G/T \simeq W_G$.
	Fact \ref{fact:WeylToAut} gives a splitting of this quotient, which depends on the choice of the symplectic forms.
\end{remark}

Let $w \in W_G$.
The automorphism $F_w$ stabilizes all elements in $L(\univ{g})$ as follows.
Since $F_w$ is $G$-equivariant, $[F_w(X), Y] = [X, Y]$ holds for any $X \in L(\lie{g})$ and $Y \in \ntDalg{G/U}$.
This shows that $F_w(X) - X \in \CC$ and hence $F_w(X) = X$ since $\lie{g}$ is semisimple (or $F_w^{|W_G|} = \id$).
The $W_G$-action on $R(\univ{t})$ is non-trivial and described easily as follows.
See \cite[Lemma 3.3]{BBP02}.

\begin{fact}\label{fact:fourier}
	Let $w \in W_G$ and set $w' = w_l^{-1}ww_l$.
	Then for any $H \in \lie{t}$, we have $F_w(R(H)) = R(w'(H) - \rho(H) + \rho(w'(H)))$.
	In particular, $F_w(\ntDalg{G/U}^\lambda) = \ntDalg{G/U}^{w'(\lambda)}$ holds for any $\lambda \in \lie{t}^*$
\end{fact}

Note that $R(\univ{t})^{W_G}$ coincides with $L(\univcent{g})$.
See also \eqref{eqn:Waction}.

The following formula is the key tool to control the difference between $\Mod(\ntDalg{G/U})$ and $\Mod_{qc}(\ntDsheaf{G/U})$.
The formula is rewritten to be convenient for our proofs.
For each $\alpha \in \Delta$, let $H_\alpha$ denote the unique element of $\lie{t}$ such that $\mu(H_{\alpha})=2(\mu, \alpha)/(\alpha, \alpha)$
holds for any $\mu \in \lie{t}^*$.

\begin{fact}[{\cite[Proposition 3.11]{BBP02}}]\label{fact:LikeGammaFactor}
	Let $\lambda$ be a dominant integral weight and set $\lambda^\vee \coloneq -w_l(\lambda)$.
	Take a basis $\set{f_i}$ of $\rring{G/U}^{\lambda}$ and its dual basis $\set{g_i}$ of $\rring{G/U}^{\lambda^\vee}$
	fixing a $G$-invariant paring.
	Then the $G\times T$-invariant operator $P_\lambda = \sum_i F_{\longw}(g_i)f_i \in \ntDalg{G/U}^{G\times T} = R(\univ{t})$
	is written as
	\begin{align*}
		P_\lambda = R\left(c \prod_{\alpha \in \proots} \prod_{i=1}^{\lambda^\vee(H_{\alpha})} (H_{\alpha} + \rho(H_{\alpha}) + i)\right)
	\end{align*}
	for some constant $c \in \CC-\set{0}$.
\end{fact}

\begin{remark}
	See \cite[Lemma 3.8]{BBP02} and references therein for $\ntDalg{G/U}^{G\times T} = R(\univ{t})$.
\end{remark}

For a $\ntDalg{G/U}$-module $M$ and $w\in W_G$, we denote by $M^{w}$ the $\ntDalg{G/U}$-module realized on the vector space $M^w=M$
with $\ntDalg{G/U}$-action $\ntDalg{G/U}\times M^w \ni (a, m) \mapsto F_w(a)m \in M^w$.

Recall that $\loc(M)$ is isomorphic to $\rsheaf{G/U}\otimes_{\rring{G/U}} M$ as an $\rsheaf{G/U}$-module
and $\loc$ is an exact functor to $\Mod_{qc}(\ntDsheaf{G/U})$ by Proposition \ref{prop:FundamentalBasicAffine}.
If the support of a $\ntDalg{G/U}$-module $M$ is contained in $\overline{G/U}-G/U$, then we have $\loc(M)=0$.

\begin{fact}[{\cite[Theorem 3.4]{BBP02}}]\label{fact:nonzeroLocalization}
	For any non-zero $\ntDalg{G/U}$-module $M$, there exists $w \in W_G$ such that $\loc(M^w) \neq 0$.
\end{fact}

\subsection{Holonomicity}\label{subsection:holonomicity}

In this subsection, we define holonomicity of $\ntDalg{G/U}$-modules.
It is well-known that the holonomicity of $\ntDalg{\CC_n}$-modules is characterized by the Gelfand--Kirillov dimension
and preserved by the Fourier transform (see e.g. \cite[Subsection 3.2.2]{HTT08}).
We will show similar properties for $\ntDalg{G/U}$-modules.

Since $\overline{G/U}$ is not smooth in general, the relation between $\Mod(\ntDalg{G/U})$ and $\Mod_{qc}(\ntDsheaf{G/U})$ is non-trivial.
Recall that $\gr\ntDsheaf{G/U}$ is isomorphic to $\rring{T^*(G/U)}$ and it is a finitely generated $\CC$-algebra by Fact \ref{fact:dalg} (3).
Combining this and Fact \ref{fact:nonzeroLocalization}, we will show that the global section functor $\sect$ on $G/U$ behaves well as for affine varieties.

We shall recall the Gelfand--Kirillov dimension and the multiplicity of modules.
We refer the reader to \cite[Chapter 8]{McRo01_noncommutative} and \cite{McSt89}.
Let $\alg{A}$ be a finitely generated $\CC$-algebra and $M$ a finitely generated $\alg{A}$-module.
Fix a finite-dimensional generating subspace $\alg{A}_1$ of $\alg{A}$ containing $\CC$,
and define a filtration on $\alg{A}$ via $\alg{A}_k \coloneq \alg{A}_1^k$, $\alg{A}_0 \coloneq \CC$.
Similarly, take a finite-dimensional generating subspace $M_0$ of $M$
and define a filtration on $M$ via $M_k\coloneq \alg{A}_k M_0$.
Such filtrations are called standard filtrations of $\alg{A}$ and $M$ in \cite{McRo01_noncommutative}.
The Gelfand--Kirillov dimension of $M$ is defined by
\begin{align*}
	\GKdim(M)(=\GKdim_{\alg{A}}(M))\coloneq \limsup_{n\rightarrow \infty}(\log \dim(M_n)/\log n).
\end{align*}
The definition does not depend on the choice of the standard filtrations.

To define the multiplicity, we recall several facts about filtrations of algebras.
See \cite[Corollary 1.4]{McSt89} for the following fact.

\begin{fact}\label{fact:GKdim}
	Let $\alg{A} = \bigcup_{n \in \NN} \alg{A}_n$ be a filtered $\CC$-algebra such that $\gr \alg{A}$ is a commutative finitely generated $\CC$-algebra.
	For any finitely generated $\alg{A}$-module $M$ with a good filtration, $\GKdim(M) = \GKdim(\gr M)$ holds.
\end{fact}

Let $\alg{A} = \bigcup_{n \in \NN} \alg{A}_n$ be a filtered $\CC$-algebra such that $\gr \alg{A}$ is a commutative finitely generated $\CC$-algebra.
Fix a standard finite-dimensional filtration of $\gr \alg{A}$.
Let $M$ be a finitely generated $\alg{A}$-module with a good filtration.
Fix a standard finite-dimensional filtration of $\gr M$.
Then there exists a polynomial $p \in \QQ[t]$ such that $p(n) = \dim((\gr M)_n)$ for any $n \gg 0$ by the theory of Hilbert polynomial.
To define the multiplicity of $M$, we need the following fact.
See \cite[Theorem 8.6.18]{McRo01_noncommutative} and Fact \ref{fact:GKdim}.

\begin{fact}\label{fact:Multiplicity}
	\begin{enumerate}
		\item $\deg p$ does not depend on the choice of the filtrations of $\alg{A}$, $\gr\alg{A}$, $M$ and $\gr M$.
		Moreover, $\deg p = \GKdim(M) = \GKdim(\gr M)$ holds.
		\item Write $m$ for the coefficient of the highest degree of $p$.
		Then $m\cdot (\deg p)! \in \NN$ holds and the integer does not depend on the choice of the filtrations of $M$ and $\gr M$.
	\end{enumerate}
\end{fact}

\begin{definition}
	Set $m(M)\coloneq m\cdot (\deg p)! \in \NN$.
	We call $m(M)$ the \define{multiplicity} (or the \define{Bernstein degree}) of $M$.
	Note that $m(M)$ may depend on the filtrations of $\alg{A}$ and $\gr\alg{A}$.
\end{definition}

The following property is standard (see \cite[Corollary 8.6.20]{McRo01_noncommutative}).

\begin{fact}\label{fact:AdditiveMultiplicity}
	Let $0\rightarrow L\rightarrow M\rightarrow N\rightarrow 0$ be an exact sequence of finitely generated $\alg{A}$-modules.
	If $\GKdim(L) = \GKdim(M) = \GKdim(N)$, then $m(M) = m(L) + m(N)$ holds.
\end{fact}

Let $\alg{B}$ be a filtered $\CC$-algebra containing $\alg{A}$.
Suppose that the inclusion $\alg{A}\hookrightarrow \alg{B}$ is a homomorphism of filtered $\CC$-algebras, $\gr \alg{B}$ is a commutative finitely generated $\CC$-algebra and the homomorphism $\gr \alg{A}\rightarrow \gr \alg{B}$ is injective.
Fix a finite-dimensional standard filtration of $\gr \alg{B}$.
The following proposition is an easy consequence of \cite[Proposition 8.4.10]{McRo01_noncommutative}.

\begin{proposition}\label{prop:BoundMultiplicity}
	Let $M$ be a finitely generated $\alg{B}$-module and $N$ a finitely generated $\alg{A}$-submodule of $M$.
	Assume that $\GKdim(N) = \GKdim(M)\eqcolon d$ and there exists a good filtration of $M$ such that the induced filtration of $N$ is a good filtration.
	Then there exists a constant $t \in \NN$ independent of $N$ and $M$ such that $m(N)\leq t^{d} m(M)$.
\end{proposition}

We shall define holonomicity of $\ntDalg{G/U}$-modules and study its properties.
Retain the notation in the previous subsection.

\begin{definition}\label{def:holonomic}
	For a $\ntDalg{G/U}$-module $M$, we say that $M$ is holonomic if $M$ is finitely generated
	and $\GKdim(M)$ is equal to $0$ or $\dim(G/U)$.
\end{definition}

As we have seen in Subsection \ref{subsection:reviewFourier}, the Fourier transform $F_w$ on $\ntDalg{G/U}$ is an automorphism for any $w \in W_G$.
By definition, the following proposition holds.

\begin{proposition}\label{prop:invariantGKdim}
	$M^w$ is also holonomic for any holonomic $\ntDalg{G/U}$-module $M$ and $w \in W_G$.
\end{proposition}

We shall show that the functors $\loc$ and $\sect$ on $G/U$ preserve the holonomicity.
For a smooth affine variety $X$ and a finitely generated $\ntDalg{X}$-module $M$, we have
\begin{align}
	\GKdim(M) = \dim(\Ch(\loc(M))), \label{eqn:GKdimCharacteristic}
\end{align}
where $\Ch(\loc(M))$ denotes the characteristic variety of $\loc(M) = \ntDsheaf{X}\otimes_{\ntDalg{X}} M$.
In fact, $\GKdim(M) = \GKdim(\gr M)$ holds for a good filtration of $M$ with respect to the order filtration of $\ntDalg{X}$ (see Fact \ref{fact:GKdim}), and $\GKdim(\gr M)$ is equal to the dimension of $\Ch(\loc(M))$
as in the proof of \cite[Proposition 3.2.11]{HTT08}.
Note that the equality (\ref{eqn:GKdimCharacteristic}) may not hold for a $\ntDalg{G/U}$-module $M$.
We can however show a modified version of the equality.

\begin{lemma}\label{lem:globalGKdim}
	Let $M$ be a finitely generated $\ntDalg{G/U}$-module.
	Then $\GKdim(M)=\max\set{\dim\Ch(\loc(M^w)): w \in W_G}$ holds.
	If the natural homomorphism $M\rightarrow \sect\circ \loc(M)$ is injective, $\GKdim(M)=\dim\Ch(\loc(M))$ holds.
\end{lemma}

\begin{remark}
	The first assertion of Lemma \ref{lem:globalGKdim} is a rewriting of \cite[Lemma 3.2]{LeSt06}.
\end{remark}

\begin{proof}
	Take a finite generating set $\set{f_1, \ldots, f_r}$ of the defining ideal of $\overline{G/U}-G/U$.
	Then each $V_i\coloneq\set{x \in \overline{G/U}: f_i(x)\neq 0}$ is contained in $G/U$, and $G/U=\bigcup_{i} V_i$ is a smooth affine open covering.

	By \cite[Lemma 3.2]{LeSt06}, we have
	\begin{align*}
		\GKdim_{\ntDalg{G/U}}(M) = \max\set{\GKdim_{\ntDalg{V_i}}(\sect(V_i, \loc(M^w))): w\in W_G, 1\leq i \leq r}.
	\end{align*}
	Hence, by (\ref{eqn:GKdimCharacteristic}), the first assertion follows.

	Set $R\coloneq \bigoplus_{i=1}^r \ntDalg{V_i}$.
	To show the second assertion, assume that $M\rightarrow \sect\circ \loc(M)$ is injective.
	Then the natural $\ntDalg{G/U}$-module homomorphism $M \rightarrow R\otimes_{\ntDalg{G/U}}M$ is injective.
	Hence we have
	\begin{align*}
		\GKdim_{\ntDalg{G/U}}(M) &\leq \GKdim_{R}(R\otimes_{\ntDalg{G/U}}M)\\
		&=\max\set{\GKdim_{\ntDalg{V_i}}(\sect(V_i, \loc(M))): 1\leq i \leq r}\\
		&=\dim\Ch(\loc(M)).
	\end{align*}
	Using the first assertion, we have $\GKdim_{\ntDalg{G/U}}(M) \geq \dim \Ch(\loc(M))$,
	and therefore we obtain the desired equality $\GKdim(M)=\dim\Ch(\loc(M))$.
\end{proof}

By Lemma \ref{lem:globalGKdim}, the Gelfand--Kirillov dimension of a non-zero $\ntDalg{G/U}$-module is greater than $\dim(G/U)$.
This is the reason why we define the holonomicity of $\ntDalg{G/U}$-modules as in Definition \ref{def:holonomic}.

\begin{corollary}\label{cor:HolonomicFiniteLength}
	Fix a finite-dimensional standard filtration of $\gr \ntDalg{G/U}$.
	Let $M$ be a holonomic $\ntDalg{G/U}$-module.
	Then the length of $M$ is less than or equal to $m(M)$.
\end{corollary}

\begin{proof}
	The assertion follows from Fact \ref{fact:AdditiveMultiplicity} and Lemma \ref{lem:globalGKdim}.
\end{proof}

We shall show that the global section functor $\sect$ preserves the holonomicity.
Remark that if $\sect(\Dmod{M})$ of a holonomic $\ntDsheaf{G/U}$-module $\Dmod{M}$ is finitely generated,
then $\sect(\Dmod{M})$ is holonomic by Lemma \ref{lem:globalGKdim}.
To show that $\sect(\Dmod{M})$ is finitely generated, we use the multiplicity.

\begin{lemma}\label{lem:SmoothAffine}
	Let $M$ be a finitely generated $\ntDalg{G/U}$-module with a good filtration, and $f \in \rring{G/U}$ such that $V\coloneq \set{x \in \overline{G/U}: f(x) \neq 0}$ is contained in $G/U$.
	Define a filtration of $M'\coloneq \rring{V}\otimes_{\rring{G/U}}M$ via $M'_k = \rring{V}\otimes_{\rring{G/U}}M_k$ ($k \in \NN$).
	Then the filtration of $M'$ is a good filtration as a $\ntDalg{V}$-module and the natural homomorphism $\rring{V}\otimes_{\rring{G/U}}\gr M \rightarrow \gr M'$ is an isomorphism of graded modules.
\end{lemma}

\begin{proof}
	Note that $V$ is a smooth affine open subset of $G/U$, and $\rring{V} = \rring{G/U}[f^{-1}]$ and $\ntDalg{V} =  \rring{V}\otimes_{\rring{G/U}}\ntDalg{G/U}$ by Proposition \ref{prop:FundamentalBasicAffine}.
	Hence $(\ntDalg{V})_k = \rring{V}\otimes_{\rring{G/U}}(\ntDalg{G/U})_k$ holds for any $k \in \NN$.
	Since the functor $ \rring{V}\otimes_{\rring{G/U}}(\cdot)$ is exact, the assertion follows.
\end{proof}

\begin{theorem}\label{thm:HolonomicLocalGlobal}
	The localization functor $\loc$ and the global section functor $\sect$ on $G/U$ preserve the holonomicity.
\end{theorem}

\begin{proof}
	Let $M$ be a holonomic $\ntDalg{G/U}$-module.
	Then $\loc(M)$ is locally finitely generated as a $\ntDsheaf{G/U}$-module and hence coherent.
	By Lemma \ref{lem:globalGKdim}, we have $\dim(\Ch(\loc(M))) \leq \dim(G/U)$ and this shows the assertion for $\loc$.

	Fix a finite-dimensional standard filtrations of $\gr \ntDalg{G/U}$ and $\gr \ntDalg{V_i}$'s to consider the multiplicity.
	Let $\Dmod{M}$ be a holonomic $\ntDsheaf{G/U}$-module and $M$ a finitely generated $\ntDalg{G/U}$-submodule of $\sect(\Dmod{M})$.
	We shall give an upper bound of $m(M)$.
	Take a finite generating set $\set{f_1, \ldots, f_r}$ of the defining ideal of $\overline{G/U}-G/U$
	and put $V_i\coloneq \set{x \in \overline{G/U}: f_i(x) \neq 0}$.
	Set $R\coloneq \bigoplus_{i} \ntDalg{V_i}$ and $M_R\coloneq R\otimes_{\ntDalg{G/U}}M = \bigoplus_{i} \left(\ntDalg{V_i}\otimes_{\ntDalg{G/U}}M\right)$.
	Since $\set{V_i}_i$ is an open covering of $G/U$, the natural homomorphism $M\rightarrow M_R$ is injective.
	By Lemma \ref{lem:globalGKdim} (and its proof), $M$ is holonomic.

	Fix a good filtration of $M$ and equip $M_R$ with the direct sum of the filtrations in Lemma \ref{lem:SmoothAffine} for $f = f_i$.
	Then the filtration of $M_R$ is a good filtration and the natural homomorphism $\gr M\rightarrow \gr M_R$ is injective.
	By Proposition \ref{prop:BoundMultiplicity}, there exists $t \in \NN$ such that
	\begin{align*}
		m(M) &\leq t^{\dim(G/U)} \sum_{i} m(\ntDalg{V_i}\otimes_{\ntDalg{G/U}}M) = t^{\dim(G/U)} \sum_{i} m(\sect(V_i, \loc(M))) \\
		&\leq t^{\dim(G/U)} \sum_{i} m(\sect(V_i, \loc\circ \sect(\Dmod{M}))) = t^{\dim(G/U)} \sum_{i} m(\sect(V_i, \Dmod{M})) \eqcolon C.
	\end{align*}
	Remark that $\loc$ is exact and the counit $\loc \circ \sect \rightarrow \id$ is a natural isomorphism by Propositoin \ref{prop:FundamentalBasicAffine}.
	Therefore, any finitely generated submodule of $\sect(\Dmod{M})$ is holonomic and its multiplicity is bounded by $C$.
	By Corollary \ref{cor:HolonomicFiniteLength}, $\sect(\Dmod{M})$ is also finitely generated and hence holonomic.
\end{proof}

Although $\loc$ is an exact functor, $\sect$ is not.
We extend Theorem \ref{thm:HolonomicLocalGlobal} to the derived functor $R\sect$.
We denote by $D^b_h(\ntDalg{G/U})$ the full subcategory of $D^b(\ntDalg{G/U})$ consisting of objects with holonomic cohomologies.
Since $\sect$ has finite cohomological dimension, the derived functor $R\sect\colon D^b(\ntDsheaf{G/U})\rightarrow D^b(\ntDalg{G/U})$ is defined.

\begin{corollary}\label{cor:HolonomicGlobalDerived}
	The derived functor $R\sect$ preserves the holonomicity, that is,
	the functor $R\sect\colon D^b_h(\ntDsheaf{G/U})\rightarrow D^b_h(\ntDalg{G/U})$ is well-defined.
\end{corollary}

\begin{proof}
	Using the truncation functors, the assertion is reduced to $H^k(G/U, \Dmod{M}) \in \Mod_h(\ntDalg{G/U})$ for any $\Dmod{M} \in \Mod_h(\ntDsheaf{G/U})$ and $k \in \NN$.
	See \cite[Definition 11.3.11]{KaSc06} for the truncation functors.

	Let $\Dmod{M}$ be a holonomic $\ntDsheaf{G/U}$-module and $V$ an affine open subset of $G/U$.
	Write $\iota\colon V\rightarrow G/U$ for the inclusion.
	By Theorem \ref{thm:HolonomicLocalGlobal}, $\sect(V, \Dmod{M}) \simeq \sect(\iota_*\iota^*(\Dmod{M}))$ is a holonomic $\ntDalg{G/U}$-module.
	Note that the inverse images and the direct images of $\ntDsheaf{}$-modules preserve the holonomicity.
	Therefore the cohomologies $H^k(G/U, \Dmod{M})$ are holonomic since they can be computed by the \v{C}ech complex.
\end{proof}

\subsection{Direct product}\label{subsection:directProduct}

We can easily generalize the results in the previous sections to twisted $\ntDsheaf{}$-modules on a variety $X\times G/U$ with some affinity condition.

Let $X$ be a smooth variety and $\Dsheaf{A}{X}$ a TDO on $X$.
Assume that the global section functor $\sect\colon \Mod_{qc}(\Dsheaf{A}{X})\rightarrow \Mod(\Dalg{A}{X})$ is exact.
In other words, any quasi-coherent $\Dsheaf{A}{X}$-module is acyclic.
Under this assumption, the category $\Mod(\Dalg{A}{X})$ is equivalent to a full subcategory of $\Mod_{qc}(\Dsheaf{A}{X})$.
See the proof of \cite[Corollary 11.2.6]{HTT08}.

\begin{fact}\label{fact:EquivalenceCategory}
	Let $\Mod_{qc}^e(\Dsheaf{A}{X})$ denote the full subcategory of $\Mod_{qc}(\Dalg{A}{X})$ whose object $\Dmod{M}$ satisfies the following conditions.
	\begin{enumerate}
		\item The canonical morphism $\Dsheaf{A}{X}\otimes_{\Dalg{A}{X}} \sect(\Dmod{M})\rightarrow \Dmod{M}$ is an epimorphism.
		\item $\sect(\Dmod{N}) \neq 0$ holds for any non-zero submodule $\Dmod{N}$ of $\Dmod{M}$.
	\end{enumerate}
	Then the global section functor gives an equivalence of categories from $\Mod_{qc}^e(\Dsheaf{A}{X})$ to $\Mod(\Dalg{A}{X})$.
	In particular, the localization $\loc(M)$ is non-zero for any non-zero $M \in \Mod(\Dalg{A}{X})$.
\end{fact}

\begin{remark}\label{remark:ExistenceEsubquotient}
	For a quasi-coherent $\Dsheaf{A}{X}$-module $\Dmod{M}$, there exists a subquotient $\Dmod{M}'$ of $\Dmod{M}$
	such that $\Dmod{M}' \in \Mod_{qc}^e(\Dsheaf{A}{X})$ and $\sect(\Dmod{M}') = \sect(\Dmod{M})$.
	In fact, letting $\Dmod{M}_0$ be the maximum submodule of $\Dmod{M}$ with $\sect(\Dmod{M}_0) = 0$
	and $\Dmod{M}_1$ the minimum submodule of $\Dmod{M}$ with $\sect(\Dmod{M}_1) = \sect(\Dmod{M})$,
	then we have $\Dmod{M}_1/\Dmod{M}_0 \in \Mod_{qc}^e(\Dsheaf{A}{X})$ and $\sect(\Dmod{M}_1/\Dmod{M}_0) \simeq \sect(\Dmod{M})$.
\end{remark}

We shall consider $\Dsheaf{A}{X\times G/U}\coloneq\Dsheaf{A}{X}\boxtimes \ntDsheaf{G/U}$, which is a TDO on $X\times G/U$.
Note that the canonical homomorphism $\sect(\Dmod{M})\otimes \sect(\Dmod{N}) \rightarrow \sect(\Dmod{M}\boxtimes \Dmod{N})$
is isomorphism for any $\Dmod{M} \in \Mod_{qc}(\Dsheaf{A}{X})$ and $\Dmod{N} \in \Mod_{qc}(\ntDsheaf{G/U})$.
This is a special case of \cite[Lemma 1.5.31]{HTT08}.
Then we have $\sect(\Dsheaf{A}{X}\boxtimes \ntDsheaf{G/U}) = \Dalg{A}{X} \otimes \ntDalg{G/U}$

We extend the Fourier transforms $(\cdot)^w$ of $\ntDalg{G/U}$-modules defined in Subsection \ref{subsection:reviewFourier} to $\Dalg{A}{X\times G/U}$-modules in the obvious way.
Combining Facts \ref{fact:nonzeroLocalization} and \ref{fact:EquivalenceCategory}, we obtain the following proposition.

\begin{proposition}
	For any $M \in \Mod(\Dsheaf{A}{X\times G/U})$, there exists $w \in W_G$ such that $\loc(M^w) \neq 0$.
\end{proposition}

Let $M$ be a finitely generated $\Dalg{A}{X\times G/U}$-module, and let $\Dmod{M} \in \Mod^e_{qc}(\Dsheaf{A}{X})$ denote the corresponding object to $M|_{\Dalg{A}{X}}$ by Fact \ref{fact:EquivalenceCategory}.
By functionality, $\Dmod{M}$ has a natural $\Dsheaf{A}{X}\otimes \ntDalg{G/U}$-module structure.

\begin{definition}\label{def:holonomic2}
	We say that $M$ is holonomic if $\GKdim(\sect(V, \Dmod{M}))$ is equal to $0$ or $\dim(X\times G/U)$
	for any affine open subset $V\subset X$.
	We denote by $\Mod_h(\Dsheaf{A}{X\times G/U})$ the full subcategory of $\Mod(\Dsheaf{A}{X\times G/U})$ consisting of holonomic modules.
\end{definition}

It is easy to see that the condition is equivalent to the existence of an affine open covering $X=\bigcup_i V_i$ such that $\GKdim(\sect(V_i, \Dmod{M}))\leq \dim(X\times G/U)$ for any $i$.
See the proof of Lemma \ref{lem:globalGKdim}.
By definition, the holonomicity is preserved by the Fourier transforms $(\cdot)^w$ (see also Proposition \ref{prop:invariantGKdim}).

\begin{proposition}
	For a holonomic $\Dalg{A}{X\times G/U}$-module $M$ and $w \in W_G$,
	$M^w$ is also holonomic.
\end{proposition}

We shall show that the global section functor preserves the holonomicity.

\begin{lemma}\label{lem:HolonomicDirectProductAffine}
	Let $\Dmod{M} \in \Mod_h(\Dsheaf{A}{X\times G/U})$.
	Suppose that $X$ is affine.
	Then $\sect(\Dmod{M})$ is holonomic.
\end{lemma}

\begin{proof}
	The proof is essentially the same as that of Theorem \ref{thm:HolonomicLocalGlobal}.
\end{proof}

\begin{theorem}\label{thm:HolonomicGlobalAX}
	Let $\Dmod{M} \in \Mod_h(\Dsheaf{A}{X\times G/U})$.
	Then $\sect(\Dmod{M})$ is holonomic.
\end{theorem}

\begin{proof}
	Write $q\colon X\times G/U\rightarrow X$ for the projection onto the first factor.
	Then $q_*(\Dmod{M})$ is an $\rsheaf{X}$-quasi-coherent $\Dsheaf{A}{X}\otimes \ntDalg{G/U}$-module with $\sect(q_*(\Dmod{M})) \simeq \sect(\Dmod{M})$.
	As we have seen in Remark \ref{remark:ExistenceEsubquotient}, there exists an $\Dsheaf{A}{X}\otimes \ntDalg{G/U}$-module subquotient $\Dmod{N}$ of $q_*(\Dmod{M})$ such that $\Dmod{N} \in \Mod^e_{qc}(\Dsheaf{A}{X})$ and $\sect(\Dmod{N}) \simeq \sect(\Dmod{M})$.

	Let $V\subset X$ be an affine open subset.
	$\sect(V, q_*(\Dmod{M})) = \sect(V\times G/U, \Dmod{M})$ is holonomic by Lemma \ref{lem:HolonomicDirectProductAffine}.
	Hence so is $\sect(V, \Dmod{N})$.
	This shows that $\sect(\Dmod{N}) (\simeq \sect(\Dmod{M}))$ is holonomic.
\end{proof}

The derived version of Theorem \ref{thm:HolonomicGlobalAX} is proved in the same way as Corollary \ref{cor:HolonomicGlobalDerived}.
Let $D^b_h(\Dalg{A}{X\times G/U})$ denote the full subcategory of $D^b(\Dalg{A}{X\times G/U})$ consisting of complexes with holonomic cohomologies.

\begin{corollary}\label{prop:HolonomicGlobalDerivedAX}
	Let $\cpx{\Dmod{M}} \in D^b_h(\Dsheaf{A}{X\times G/U})$.
	Then $R\sect(\cpx{\Dmod{M}}) \in D^b_h(\Dalg{A}{X\times G/U})$ holds.
\end{corollary}

\section{Generalization of Beilinson--Bernstein correspondence}

The goal of this section is to study the relation between the category of $\ntDsheaf{}$-modules
on $G/U$ and the category of twisted $\ntDsheaf{}$-modules on the full flag variety $G/B$.
We prove a generalization of the Beilinson--Bernstein correspondence.

\subsection{Beilinson--Bernstein correspondence}\label{subsection:BB}

We recall the Beilinson--Bernstein correspondence.
We refer the reader to \cite[Subsection 2.5, 3.2 and 3.3]{BeBe93} and \cite[Theorems 11.2.2, 11.2.3 and 11.2.4]{HTT08}.

Let $G$ be a connected simply-connected semisimple algebraic group.
Fix a Borel subgroup $B=TU$ with a maximal torus $T$ and unipotent radical $U$.
Let $\roots$ and $\proots$ denote the sets of roots and positive roots determined by $G, T$ and $B$.
Set $\rho \coloneq \frac{1}{2}\sum_{\alpha \in \proots} \alpha$.
Write $p\colon G/U\rightarrow G/B$ for the natural projection.

For any open subset $V \subset G/B$, $\sect(V, p_* \ntDsheaf{G/U})$ is a $T$-module.
Hence $(p_* \ntDsheaf{G/U})^T$ is a sheaf of algebras on $G/B$.
By \cite[Proposition 8]{BoBr89}, we have
\begin{align}
	\sect((p_* \ntDsheaf{G/U})^T) \simeq \univ{g}\otimes_{\univcent{g}}\univ{t}. \label{eqn:sapovalov}
\end{align}
See \eqref{eqn:Waction} for the homomorphism $\univcent{g}\hookrightarrow \univ{t}$.
For a weight $\lambda \in \lie{t}^*$, we set
\begin{align*}
	\ntDsheaf{G/B,\lambda}\coloneq (p_* \ntDsheaf{G/U})^T\otimes_{\univ{t}} \End_{\CC}(\CC_{\lambda + \rho}),
\end{align*}
which is a $G$-equivariant TDO on $G/B$.
Then we have
\begin{align*}
	\ntDsheaf{G/B} \simeq \ntDsheaf{G/B, -\rho},\quad \ntDsheaf{G/B,\lambda}^{op} &\simeq \ntDsheaf{G/B, -\lambda},\\
	\Dmod{L}_{\mu}\otimes_{\rsheaf{G/B}}\ntDsheaf{G/B, \lambda}\otimes_{\rsheaf{G/B}}\Dmod{L}_{\mu}^{\invshf} &\simeq \ntDsheaf{G/B, \lambda+\mu}
	\quad (\mu \text{: integral weight}),
\end{align*}
and $\ntDsheaf{G/B, \lambda}$ acts on $\Dmod{L}_{\lambda + \rho} = (p_*\rsheaf{G/U}\otimes \CC_{\lambda + \rho})^T$ if $\lambda$ is integral.

\begin{fact}[Beilinson--Bernstein \cite{BeBe81}]\label{fact:beilinsonBernsten}
	Let $\lambda \in \lie{t}^*$.
	\begin{enumerate}
		\item The homomorphism $\univ{g}\rightarrow \sect(\ntDsheaf{G/B,\lambda})$ is surjective and its kernel is the minimal primitive ideal with the infinitesimal character $\lambda$.
		\item If $\lambda$ is anti-dominant, the global section functor $\sect\colon\Mod_{qc}(\ntDsheaf{G/B, \lambda}) \rightarrow \Mod(\ntDalg{G/B, \lambda})$
		is exact.
		\item If $\lambda$ is regular and anti-dominant, the global section functor $\sect$
		gives an equivalence of categories.
	\end{enumerate}
\end{fact}

\subsection{Support as a \texorpdfstring{$\univ{t}$}{U(t)}-module}

In this subsection, we deal with a relationship between the support of a $\ntDalg{G/U}$-module as an $\rring{\overline{G/U}}$-module
and as a $\univ{t}$-module using Fact \ref{fact:LikeGammaFactor}.
As an application, we will prove a vanishing theorem of $\Tor$.
Retain the notation in the previous subsection.
Recall $\rring{G/U}=\rring{\overline{G/U}}$.

Let $S$ be a multiplicative subset of $\univ{t}$ generated by
\begin{align}
	\set{H_{\alpha} + \rho(H_{\alpha}) + i: \alpha \in \proots, i \text{ is a positive integer}}, \label{eqn:S}
\end{align}
where $H_{\alpha}$ is the unique element of $\lie{t}$ with $\mu(H_\alpha) = 2(\mu, \alpha)/(\alpha, \alpha)$ for any $\mu \in \lie{t}^*$ (c.f.\ Fact \ref{fact:LikeGammaFactor}).

\begin{lemma}\label{lem:supportUt}
	Let $M$ be a $\ntDalg{G/U}$-module.
	Assume that the support of $M$ as an $\rring{\overline{G/U}}$-module is contained in $\overline{G/U}-G/U$.
	Then $S^{-1}M = 0$ holds.
\end{lemma}

\begin{remark}
	We have two actions of $\univ{t}$ on $M$, that is, $(L, \univ{t})$ and $(R, \univ{t})$ (see \eqref{eqn:ActionGT}).
	Hereafter, the localization functor $S^{-1}$ is taken with respect to the action $(R, \univ{t})$.
\end{remark}

\begin{proof}
	Let $m \in M$.
	It is enough to find $P \in S$ with $Pm = 0$.
	It is well-known that the defining ideal of $\overline{G/U}-G/U$ is generated by $\rring{G/U}^{\rho}$.
	Since the support of $M$ is contained in $\overline{G/U}-G/U$, there exists $n \in \NN$ such that
	$\rring{G/U}^{n\rho}m = 0$.
	Take $P_{n\rho}$ as in Fact \ref{fact:LikeGammaFactor}.
	By the construction of $P_{n\rho}$, we have $P_{n\rho}\in F_{\longw}(\rring{G/U}^{n\rho})\rring{G/U}^{n\rho}$.
	Hence $P_{n\rho}m = 0$ holds.

	By Fact \ref{fact:LikeGammaFactor}, $P_{n\rho}$ is written as
	\begin{align*}
		P_{n\rho} = c R\left(\prod_{\alpha \in \proots} \prod_{i=1}^{n\rho(H_{\alpha})} (H_{\alpha} + \rho(H_{\alpha}) + i)\right)
	\end{align*}
	for some $c \in \CC^\times$.
	Hence we have $c^{-1}P_{n\rho}\in R(S)$ and this shows the assertion.
\end{proof}

If $\lambda$ is anti-dominant, then $-(\lambda + \rho)(H_{\alpha}) + \rho(H_{\alpha}) + i$
is non-zero for any $\alpha \in \proots$ and positive integer $i$.
Hence the right $\univ{t}$-module $\CC_{\lambda+\rho}$ can be regarded as an $S^{-1}\univ{t}$-module if $\lambda$ is anti-dominant.
(Remark that $\CC_{\lambda}$ is regarded as a right $\univ{t}$-module via $1 \cdot H = \lambda({}^t\! H)$.)
The following corollary is therefore a direct consequence of Theorem \ref{lem:supportUt}.

\begin{corollary}\label{cor:vanishingTor}
	Let $M$ be a $\ntDalg{G/U}$-module and $\lambda \in \lie{t}^*$ an anti-dominant weight.
	If the support of $M$ as an $\rring{\overline{G/U}}$-module is contained in $\overline{G/U}-G/U$,
	then $\Tor^{\univ{t}}_i(\CC_{\lambda+\rho}, M)=0$ holds for any $i \in \ZZ$.
\end{corollary}

\begin{proof}
	Since $\CC_{\lambda+\rho}$ is a right $S^{-1}\univ{t}$-module, we have
	\begin{align*}
		\Tor^{\univ{t}}_i(\CC_{\lambda+\rho}, M) \simeq \Tor^{S^{-1}\univ{t}}_i(\CC_{\lambda+\rho}, S^{-1}M) = 0
	\end{align*}
	for any $i \in \ZZ$.
\end{proof}

Recall that the functor $\loc$ is exact and the counit $\loc \circ \sect \rightarrow \id$ is a natural isomorphism (Proposition \ref{prop:FundamentalBasicAffine}).
In general, the unit $\eta\colon \id \rightarrow \sect \circ \loc$ is not a natural isomorphism.
As an application of Theorem \ref{lem:supportUt}, we see that the difference between $M$ and $\sect \circ \loc(M)$
is `small' as $\univ{t}$-modules for a $\ntDalg{G/U}$-module $M$.

\begin{corollary}\label{cor:localAndGlobal}
	Let $M$ be a $\ntDalg{G/U}$-module.
	Then the $S^{-1}\ntDalg{G/U}^T$-module homomorphism $\eta_{S}\colon S^{-1}M \rightarrow S^{-1}\sect\circ\loc(M)$ induced from the unit $\eta$ is an isomorphism.
	If $\lambda \in \lie{t}^*$ is an anti-dominant weight, the unit $\eta$ induces a $\lie{g}$-module isomorphism
	\begin{align*}
		\eta_{\lambda}\colon \Tor^{\univ{t}}_i(\CC_{\lambda+\rho}, M) \xrightarrow{\simeq} \Tor^{\univ{t}}_i(\CC_{\lambda+\rho}, \sect\circ\loc(M))
	\end{align*}
	for any $i \in \NN$.
\end{corollary}

\begin{proof}
	As we have seen in the above of Corollary \ref{cor:vanishingTor}, $\CC_{\lambda+\rho}$ is a right $S^{-1}\univ{t}$-module.
	Hence the second assertion follows from the first assertion.
	We shall show that $\eta_S$ is bijective.

	To see this, consider the exact sequence
	\begin{align}
		0 \rightarrow \Ker(\eta) \rightarrow M \xrightarrow{\eta} \sect\circ \loc(M) \rightarrow \Coker(\eta) \rightarrow 0.
		\label{eqn:localAndGlobalExactSequence}
	\end{align}
	We apply the exact functor $\loc$ to the exact sequence and then obtain an exact sequence
	\begin{align*}
		0 \rightarrow \loc(\Ker(\eta)) \rightarrow \loc(M) \xrightarrow{\loc(\eta)} \loc\circ \sect\circ \loc(M) \rightarrow \loc(\Coker(\eta)) \rightarrow 0.
	\end{align*}
	Since the counit $\loc \circ \sect \rightarrow \id$ is a natural isomorphism, 
	$\loc(\eta)$ is an isomorphism.
	Hence we have $\loc(\Ker(\eta)) = \loc(\Coker(\eta)) = 0$.
	This implies that the supports of $\Ker(\eta)$ and $\Coker(\eta)$ are contained in $\overline{G/U}-G/U$.
	From Theorem \ref{lem:supportUt}, we obtain $S^{-1}\Ker(\eta) = S^{-1}\Coker(\eta) = 0$.
	This and the exact sequence \eqref{eqn:localAndGlobalExactSequence} show that $\eta_S$ is bijective.
\end{proof}

\subsection{Vanishing theorem}

In this subsection, we prove a vanishing theorem for $S^{-1}(p_* \ntDsheaf{G/U})^T$-modules.
This is a generalization of the exactness of $\sect$ in the Beilinson--Bernstein correspondence (Fact \ref{fact:beilinsonBernsten}).

We set
\begin{align*}
	\ntDsheaf{G/B,S}\coloneq S^{-1}(p_* \ntDsheaf{G/U})^T.
\end{align*}
Then $\ntDsheaf{G/B,S}$ is a sheaf of algebras since $R(\univ{t})$ is contained in the center of $(p_* \ntDsheaf{G/U})^T$.
By $\rsheaf{G/B} \simeq (p_*\rsheaf{G/U})^T$, there is a monomorphism $\rsheaf{G/B}\hookrightarrow \ntDsheaf{G/B,S}$ and $\ntDsheaf{G/B,S}$ is a quasi-coherent left/right $\rsheaf{G/B}$-module.
By (\ref{eqn:sapovalov}), we have
\begin{align*}
	(\ntDalg{G/B,S}\coloneq)\sect(\ntDsheaf{G/B,S}) \simeq \univ{g}\otimes_{\univcent{g}}S^{-1}\univ{t}.
\end{align*}
If $\lambda$ is anti-dominant, there is a natural epimorphism $\ntDsheaf{G/B, S}\rightarrow \ntDsheaf{G/B, \lambda}$
of sheaves of algebras.

We say that a $(p_* \ntDsheaf{G/U})^T$-module is quasi-coherent if the module is quasi-coherent as an $\rsheaf{G/B}$-module.
For a quasi-coherent $(p_* \ntDsheaf{G/U})^T$-module $\Dmod{M}$, the cohomology $H^i(G/B, \Dmod{M})$ may not vanish for $i > 0$.
This is because $(p_* \ntDsheaf{G/U})^T$ is so large that $\Mod_{qc}((p_* \ntDsheaf{G/U})^T)$ contains $\Mod_{qc}(\ntDsheaf{G/B,\lambda})$ 
with `bad' $\lambda$ for which the Beilinson--Bernstein correspondence does not hold.
The algebra $\ntDsheaf{G/B,S}$ knows, in contrast, only $\ntDsheaf{G/B, \lambda}$ with anti-dominant $\lambda$.
We can therefore expect that an analogue of the Beilinson--Bernstein correspondence
holds for $\ntDsheaf{G/B,S}$.
In fact, we prove the following theorem (see also Theorem \ref{thm:Daffine}).

\begin{theorem}\label{thm:vanishing}
	Let $\Dmod{M}$ be a quasi-coherent $\ntDsheaf{G/B,S}$-module.
	Then we have $H^i(G/B, \Dmod{M}) = 0$ for any $i > 0$.
\end{theorem}

The outline of the proof is the same as that of the Beilinson--Bernstein correspondence in \cite{BeBe81}.
The only difference is how to construct a splitting in the following lemma.
For a dominant integral weight $\mu \in \lie{t}^*$, we denote by $F(\mu)$ a finite-dimensional irreducible $G$-module with the highest weight $\mu$.

\begin{lemma}\label{lemma:splitting}
	Let $\Dmod{M}$ be a quasi-coherent $\ntDsheaf{G/B,S}$-module
	and $\mu \in \lie{t}^*$ an anti-dominant integral weight.
	Then there exists an $\rsheaf{G/B}$-module homomorphism
	\begin{align*}
		i_{\mu}\colon \Dmod{M} \rightarrow F(-\mu)\otimes \Dmod{L}_{\mu}\otimes_{\rsheaf{G/B}} \Dmod{M}
	\end{align*}
	and $i_{\mu}$ has a left inverse as a morphism of sheaves of abelian groups.
\end{lemma}

\begin{proof}
	Set $F\coloneq \sect(\Dmod{L}_{-w_l(\mu)})$ and identify $F$ with $F(-\mu)$.
	Take a basis $\set{f_i} \subset \sect(\Dmod{L}_{\mu})\simeq F(w_{l}(\mu))$
	and a dual basis $\set{g_i} \subset F$ fixing a $G$-invariant pairing.
	We define $i_\mu$ via $i_{\mu}(m) = \sum_{i} g_i \otimes f_i \otimes m$ for a local section $m \in \Dmod{M}$.
	By definition, $i_\mu$ is an $\rsheaf{G/B}$-module homomorphism.

	We shall construct a left inverse of $i_{\mu}$.
	Through $\Dmod{L}_{\mu} \xrightarrow{\simeq} (p_* \rsheaf{G/U})^{-\mu} \rightarrow (p_{*} \ntDsheaf{G/U})^{-\mu}$,
	we identify $\Dmod{L}_{\mu}$ with an $\rsheaf{G/B}$-submodule of $(p_{*} \ntDsheaf{G/U})^{-\mu}$.
	We have a non-zero $G$-module homomorphism
	$F \simeq \rring{G/U}^{\longw(\mu)} \xrightarrow{F_{w_l}} \ntDalg{G/U}^{\mu}$.
	Then we have a right $\rsheaf{G/B}$-module homomorphism
	\begin{align*}
		\iota\colon F(-\mu)\otimes\Dmod{L}_{\mu} \rightarrow \ntDalg{G/U}^{\mu}\otimes(p_{*} \ntDsheaf{G/U})^{-\mu}
		\rightarrow (p_* \ntDsheaf{G/U})^T,
	\end{align*}
	where the last morphism is the multiplication map in $p_* \ntDsheaf{G/U}$.
	We therefore obtain a morphism of sheaves of abelian groups
	\begin{align*}
		s\colon F(-\mu)\otimes\Dmod{L}_{\mu}\otimes_{\rsheaf{G/B}} \Dmod{M}
		\xrightarrow{\iota\otimes \id} (p_* \ntDsheaf{G/U})^T\otimes_{\rsheaf{G/B}} \Dmod{M}
		\xrightarrow{\alpha} \Dmod{M},
	\end{align*}
	where the last morphism $\alpha$ is given by the action of $\ntDsheaf{G/B,S}$ on $\Dmod{M}$.

	We shall show that $s\circ i_\mu$ is an isomorphism.
	For a local section $m \in \Dmod{M}$, we have
	\begin{align*}
		s\circ i_\mu(m) &= \alpha \circ (\iota\otimes \id) \circ i_{\mu}(m)\\
		&= \alpha \circ (\iota\otimes \id) \left(\sum_i g_i\otimes f_i \otimes m\right)\\
		&=\sum_i \alpha (F_{w_l}(g_i)f_i\otimes m)\\
		&=\left(\sum_i F_{w_l}(g_i)f_i\right) m\\
		&=c\prod_{\alpha \in \proots} \prod_{i=1}^{-\mu^\vee(H_{\alpha})} (H_{\alpha} + \rho(H_{\alpha}) + i)m,
	\end{align*}
	where the last equality comes from Fact \ref{fact:LikeGammaFactor}.
	The factors $H_{\alpha} + \rho(H_{\alpha}) + i$ in the last equality 
	are invertible in $\ntDsheaf{G/B,S}$.
	This therefore implies that $s\circ i_\mu$ is an isomorphism.
	We have proved the lemma.
\end{proof}

The remaining part of the proof of Theorem \ref{thm:vanishing} is the same as that in \cite{BeBe81}.
We however give the proof for reader's convenience.

\begin{proof}[Proof of Theorem \ref{thm:vanishing}]
	Fix an integer $i > 0$.
	Let $\Dmod{N}$ be a coherent $\rsheaf{G/B}$-submodule of $\Dmod{M}$.
	Consider the commutative diagram
	\begin{align*}
		\xymatrix{
		\Dmod{N} \ar@{^{(}->}[r]\ar[d] & \Dmod{M} \ar[d]_{i_{\mu}} \\
		F(\mu)\otimes \Dmod{L}_{\mu} \otimes_{\rsheaf{G/B}}\Dmod{N} \ar@{^{(}->}[r]&
		F(\mu)\otimes \Dmod{L}_{\mu} \otimes_{\rsheaf{G/B}}\Dmod{M} 
		}
	\end{align*}
	for any anti-dominant integral weight $\mu\in \lie{t}^*$.
	By Lemma \ref{lemma:splitting}, $i_\mu$ has a left inverse $s$.
	Since $\Dmod{N}$ is a coherent $\rsheaf{G/B}$-module and $\Dmod{L}_{-\rho}$ is ample,
	there exists a positive integer $n$ such that $H^i(G/B, \Dmod{L}_{-n\rho} \otimes_{\rsheaf{G/B}}\Dmod{N}) = 0$.
	Then we have
	\begin{align*}
		H^i(G/B, F(n\rho)\otimes\Dmod{L}_{-n\rho} \otimes_{\rsheaf{G/B}}\Dmod{N}) &\simeq F(n\rho)\otimes H^i(G/B, \Dmod{L}_{-n\rho} \otimes_{\rsheaf{G/B}}\Dmod{N})\\
		&= 0.
	\end{align*}
	By the above commutative diagram and the existence of the right inverse $s$, the homomorphism $H^i(G/B, \Dmod{N}) \rightarrow H^i(G/B, \Dmod{M})$ induced from the inclusion is the zero map.
	This implies
	\begin{align*}
		H^i(G/B, \Dmod{M}) \simeq \lim_{\xrightarrow[\Dmod{N}]{}}H^i(G/B, \Dmod{N}) = 0.
	\end{align*}
	This completes the proof.
\end{proof}

\begin{remark}
	If $\lambda$ is an anti-dominant weight, the category $\Mod_{qc}(\ntDsheaf{G/B,\lambda})$
	can be regarded as a full subcategory of $\Mod_{qc}(\ntDsheaf{G/B,S})$.
	Hence Theorem \ref{thm:vanishing} is a generalization of the exactness of $\sect$ in the Beilinson--Bernstein correspondence.
\end{remark}

\subsection{\texorpdfstring{$\ntDsheaf{G/B,S'}$}{D(G/B,S')}-affinity}

Let $S'$ be a multiplicative subset of $\univ{t}$ generated by
\begin{align*}
	\set{H_{\alpha} + \rho(H_{\alpha}) + i: \alpha \in \proots, i \text{ is a non-negative integer}},
\end{align*}
(see \eqref{eqn:S} for $H_\alpha$).
Set $\ntDsheaf{G/B,S'}\coloneq(S')^{-1}(p_* \ntDsheaf{G/U})^T$ and $\ntDalg{G/B,S'}\coloneq \sect(\ntDsheaf{G/B,S'})$.
By definition, we have $S'\supset S$.

In this subsection, we prove that $G/B$ is $\ntDsheaf{G/B,S'}$-affine, i.e.\ 
the global section functor gives an equivalence of categories from $\Mod_{qc}(\ntDsheaf{G/B,S'})$ to $\Mod(\ntDalg{G/B,S'})$.
We have already shown that the functor $\sect$ is exact in Theorem \ref{thm:vanishing}.
By \cite[Proposition 1.4.4]{HTT08}, it is enough to show that $\sect(\Dmod{M})$ is non-zero
for any non-zero $\Dmod{M} \in \Mod_{qc}(\ntDsheaf{G/B,S'})$.
We reduce this to the Beilinson--Bernstein correspondence.
To show this, we prove the following lemma.

\begin{lemma}\label{lem:nonVanishingQuot}
	Let $\alg{A}$ be a finitely generated integral domain over $\CC$ and $T$ an at most countable multiplicative subset of $\alg{A}$.
	Let $M$ be a non-zero torsion-free $T^{-1}\alg{A}$-module of at most countable dimension.
	Then there exists a maximal ideal $I$ of $T^{-1}\alg{A}$ such that $M/IM\neq 0$.
\end{lemma}

\begin{proof}
	Let $K$ denote the field of fractions of $\alg{A}$.
	Since $M$ is a torsion-free $T^{-1}\alg{A}$-module, 
	the canonical homomorphism $M \rightarrow K\otimes_{T^{-1}\alg{A}}M$
	is injective.
	Since $K\otimes_{T^{-1}\alg{A}}M$ is a vector space over $K$ and $M$ has at most countable dimension, there exists an at most countable multiplicative subset $T'$ of $\alg{A}$ such that $T'\supset T$ and $(T')^{-1}M$ is a free $(T')^{-1}\alg{A}$-module.
	Since $T'$ is countable, there is a maximal ideal $I$ of $T^{-1}\alg{A}$ such that 
	$I\cap T' = \emptyset$.
	Then we have $M/IM \simeq (T')^{-1}M / I(T')^{-1}M \neq 0$.
\end{proof}

The following proposition is proved in the same way as \cite[Corollary 1.4.17]{HTT08}.
We omit the proof.

\begin{proposition}\label{prop:ExistenceCoherentSub}
	Let $\Dmod{M} \in \Mod_{qc}(\ntDsheaf{G/B,S'})$, and $V$ an open subset of $G/B$.
	For any coherent submodule $\Dmod{N}$ of $\Dmod{M}|_V$, there exists a coherent submodule $\widetilde{\Dmod{N}}$
	such that $\widetilde{\Dmod{N}}|_V\simeq \Dmod{N}$.
\end{proposition}

\begin{lemma}\label{lem:nonVanishing}
	$\sect(\Dmod{M})$ is non-zero for any non-zero $\Dmod{M} \in \Mod_{qc}(\ntDsheaf{G/B,S'})$.
\end{lemma}

\begin{proof}
	By Proposition \ref{prop:ExistenceCoherentSub}, we can assume that $\Dmod{M}$ is a coherent $\ntDsheaf{G/B,S'}$-module.
	Since $(S')^{-1}\univ{t}$ is noetherian, there is a maximal element $J$ of
	\begin{align*}
		\set{I \subset (S')^{-1}\univ{t}: \Dmod{M}^I \neq 0, I \text{ is an ideal}},
	\end{align*}
	where $\Dmod{M}^I$ is the $\ntDsheaf{G/B,S'}$-submodule of $\Dmod{M}$ annihilated by $I$.
	By the maximality of $J$, $J$ is a prime ideal and $\sect(V, \Dmod{M}^J)$ is a torsion-free $(S')^{-1}\univ{t}/ J$-module
	for any open subset $V \subset G/B$ with $\sect(V, \Dmod{M}^J) \neq 0$.
	Replacing $\Dmod{M}$ with $\Dmod{M}^J$, we can assume $\Dmod{M}=\Dmod{M}^J$.

	Fix an affine open subset $V$ of $G/B$ such that $\sect(V, \Dmod{M})\neq 0$ and $p^{-1}(V) \simeq V\times T$.
	Then $\sect(V, \Dmod{M})$ is a finitely generated $\ntDalg{V}\otimes (S')^{-1}\univ{t}$-module
	and torsion-free as an $(S')^{-1}\univ{t}/J$-module.
	Hence, applying Lemma \ref{lem:nonVanishingQuot}, we can take a regular anti-dominant weight $\lambda \in \lie{t}^*$ such that $\CC_{\lambda + \rho}\otimes_{\univ{t}} \sect(V, \Dmod{M}) \neq 0$.
	Note that $\CC_{\lambda + \rho}$ is a right $(S')^{-1}\univ{t}$-module.

	$\CC_{\lambda + \rho}\otimes_{\univ{t}}\Dmod{M}$ is a quotient module of $\Dmod{M}$
	and we have $\sect(V, \CC_{\lambda + \rho}\otimes_{\univ{t}}\Dmod{M}) \simeq \CC_{\lambda + \rho}\otimes_{\univ{t}}\sect(V, \Dmod{M})$ since $V$ is affine.
	Moreover, $\CC_{\lambda + \rho}\otimes_{\univ{t}}\Dmod{M}$ is a quasi-coherent $\ntDsheaf{G/B,\lambda}$-module.
	By the Beilinson--Bernstein correspondence (Fact \ref{fact:beilinsonBernsten}), $\sect(\CC_{\lambda + \rho}\otimes_{\univ{t}}\Dmod{M}) \neq 0$ holds.
	Since $\sect$ is an exact functor by Theorem \ref{thm:vanishing}, we obtain $\sect(\Dmod{M}) \neq 0$.
\end{proof}

By \cite[Proposition 1.4.4, Proposition 1.4.13]{HTT08}, Theorem \ref{thm:vanishing} and Lemma \ref{lem:nonVanishing},
we have shown the following theorem.

\begin{theorem}\label{thm:Daffine}
	$G/B$ is $\ntDsheaf{G/B, S'}$-affine, that is,
	the global section functor
	\begin{align*}
		\sect\colon \Mod_{qc}(\ntDsheaf{G/B, S'})&\rightarrow \Mod(\univ{g}\otimes_{\univcent{g}}(S')^{-1}\univ{t})
	\end{align*}
	gives an equivalence of categories.
\end{theorem}

\subsection{Direct image and coinvariant}\label{subsection:directImageCoinv}

In this subsection, we study the direct image functor
\begin{align*}
	Dp_{+, \lambda}\colon D^b_{qc}(\ntDsheaf{G/U}) \rightarrow D^b_{qc}(\ntDsheaf{G/B,\lambda})
\end{align*}
using Theorem \ref{thm:vanishing}.
Retain the notation in the previous subsection.
In particular, we use the multiplicative subset $S$ defined in \eqref{eqn:S}.
Throughout this subsection, fix an anti-dominant weight $\lambda \in \lie{t}^*$.

Since $p\colon G/U \rightarrow G/B$ is a principal $T$-bundle, the direct image functor can be computed by the tensor product functor $\CC_{\lambda+\rho}\otimes_{\univ{t}}(\cdot)$ by Fact \ref{fact:DirectImagePrincipal}.
We prove that this is also true for the global section.

\begin{theorem}\label{thm:directImageAndCoinvariant}
	Let $\Dmod{M}\in \Mod_{qc}(\ntDsheaf{G/U})$.
	Then there exists a natural $\lie{g}$-module isomorphism
	$H^i(R\sect(Dp_+(\Dmod{M}))) \simeq \Tor_{-i}^{\univ{t}}(\CC_{\lambda + \rho}, \sect(\Dmod{M}))$.
\end{theorem}

\begin{proof}
	Note that $\Dmod{M}$ has a locally projective resolution $\cpx{\Dmod{F}} \in D^b_{qc}(\ntDsheaf{G/U})$ (see \cite[Corollary 1.4.20]{HTT08}).
	Since $p\colon G/U\rightarrow G/B$ is a principal $T$-bundle in the Zariski topology, there exists an affine open covering $G/B = \bigcup_i V_i$ such that $\sect(V_i, p_*(\ntDsheaf{G/U}))$ is free as a $\univ{t}$-module.
	This implies that $\sect(V_i, \cpx{\Dmod{F}})$'s are complexes of projective $\univ{t}$-modules.
	Hence taking a free resolution $\cpx{E}$ of the right $S^{-1}\univ{t}$-module $\CC_{\lambda + \rho}$, we have isomorphisms
	\begin{align*}
		\CC_{\lambda + \rho} \otimes_{\univ{t}} p_*(\cpx{\Dmod{F}})& \simeq \CC_{\lambda + \rho} \otimes_{S^{-1}\univ{t}} S^{-1}p_*(\cpx{\Dmod{F}}) \\
		&\simeq \cpx{E} \otimes_{S^{-1}\univ{t}} S^{-1}p_*(\Dmod{M}) \\
		&\simeq \cpx{E} \otimes_{\univ{t}} p_*(\Dmod{M})
	\end{align*}
	in $D^b_{qc}(\ntDsheaf{G/B,S})$.
	Therefore the following proposition shows the assertion.
\end{proof}

\begin{proposition}\label{prop:tensorVSGlobal}
	Let $\Dmod{M}$ be a quasi-coherent $(p_*\ntDsheaf{G/U})^T$-module and $N$ an $S^{-1}\univ{t}$-module with a free resolution $\cpx{F}$.
	Then there exists a natural isomorphism
	\begin{align*}
		\cpx{F}\otimes_{\univ{t}}\sect(\Dmod{M}) \simeq R\sect(\cpx{F}\otimes_{\univ{t}}\Dmod{M})
	\end{align*}
	in $D^b(\ntDalg{G/B, S})$.
\end{proposition}

\begin{proof}
	Since $S^{-1}$ is an exact functor and $F^i$'s are $S^{-1}\univ{t}$-modules, we have isomorphisms
	\begin{align*}
		\cpx{F}\otimes_{\univ{t}}\sect(S^{-1}\Dmod{M}) &\simeq \cpx{F}\otimes_{\univ{t}}S^{-1}\sect(\Dmod{M}) \simeq \cpx{F}\otimes_{\univ{t}}\sect(\Dmod{M}), \\
		\cpx{F}\otimes_{\univ{t}} S^{-1}\Dmod{M} &\simeq \cpx{F}\otimes_{\univ{t}}\Dmod{M}.
	\end{align*}
	We can therefore assume that $\Dmod{M}$ is a quasi-coherent $\ntDsheaf{G/B,S}$-module.

	By Theorem \ref{thm:vanishing}, $\Dmod{M}$ is acyclic.
	Hence $\cpx{F}\otimes_{\univ{t}} \Dmod{M}$ is a complex of acyclic sheaves.
	This implies
	\begin{align*}
		R\sect(\cpx{F}\otimes_{\univ{t}} \Dmod{M}) = \sect(\cpx{F}\otimes_{\univ{t}} \Dmod{M}) \simeq \cpx{F}\otimes_{\univ{t}} \sect(\Dmod{M}).
	\end{align*}
	We have shown the assertion.
\end{proof}

Combining Theorem \ref{thm:directImageAndCoinvariant} and Corollary \ref{cor:localAndGlobal}, we obtain the following result.

\begin{corollary}\label{cor:localizationCoinv}
	Let $M$ be a $\ntDalg{G/U}$-module.
	Then there exists a natural $\lie{g}$-module isomorphism
	$H^i\circ R\sect\circ Dp_{+} \circ \loc(M) \simeq \Tor_{-i}^{\univ{t}}(\CC_{\lambda+\rho}, M)$.
\end{corollary}

\newcommand{\uVerm}{\mathbb{V}}

\section{\texorpdfstring{$\univ{t}$}{U(t)}-support of \texorpdfstring{$V/\lie{u}V$}{V/nV}}

In this section, we deal with the $\lie{t}$-module $V/\lie{u}V$ of a $\lie{g}$-module $V$.
We show that if a `good' algebra $\alg{A}$ acts on $V$, some large algebra acts on $V/\lie{u}V$.
To construct the action, we use $\ntDsheaf{G/U}$.

\subsection{Cartan subalgebras for \texorpdfstring{$\lie{g}$}{g}-modules}

We have introduced the notion of (small) Cartan subalgebras for $\lie{g}$-modules in \cite{Ki24-2}.
We recall the definition and give a different proof for the abelian case.

Let $G$ be a connected reductive algebraic group.
Fix a Borel subgroup $B=TU$ of $G$ with a maximal torus $T$ and unipotent radical $U$.
Write $W_G$ for the Weyl group of $G$ and $q\colon \lie{t}^* \rightarrow \lie{t}^*/W_G$ for the quotient map.
Let $\rho$ denote half the sum of all roots in $\lie{u}$.

\begin{definition}
	Let $V$ be a non-zero $\lie{g}$-module.
	We set
	\begin{align*}
		\rfilt{k}{V} &\coloneq \set{v \in V: \Variety(\Ann_{\univcent{g}}(v)) \leq k} \quad (k \in \NN), \\
		\Rrankmax(V) &\coloneq \min\set{k \in \NN: \rfilt{k}{V} = V}.
	\end{align*}

	We say that $V$ has a small Cartan subalgebra if there exists a subspace $\lie{a}^*\subset \lie{t}^*$ such that
	any $\Variety(P)$ ($P \in \Ass_{\univcent{g}}(V)$) is of the form $q(\lie{a}^* + \mu)$ ($\mu \in \lie{t}^*$).
	We call the quotient $\lie{t}/\bigcap_{\lambda \in \lie{a}^*} \Ker(\lambda)$ a \define{small Cartan subalgebra} for $V$.

	We say that $V$ has a Cartan subalgebra if there exists a subspace $\lie{c}^*\subset \lie{t}^*$ such that
	any irreducible component of $\Variety(\Ann_{\univcent{g}}(V))$ is of the form $q(\lie{c}^* + \mu)$ ($\mu \in \lie{t}^*$).
	We call the quotient $\lie{t}/\bigcap_{\lambda \in \lie{c}^*} \Ker(\lambda)$ a \define{Cartan subalgebra} for $V$.
\end{definition}

Let $(\alg{A}, G)$ be a generalized pair.
We have shown in \cite[Theorem 3.25]{Ki24-2} that any irreducible $\alg{A}$-module of at most countable dimension has a small Cartan subalgebra as a $\lie{g}$-module.
The following result in \cite[Proposition 3.38 and its proof]{Ki24-2} is one of the motivations to study $V/\lie{u}V$.
Let $\theta$ be an involution of $G$ and $K$ a (connected and finite) covering of $(G^\theta)_0$.

\begin{fact}\label{fact:ReductionToUcoinv}
	Let $V$ be an irreducible $(\alg{A}, K)$-module of at most countable dimension.
	Set $n \coloneq \Rrankmax(V|_{\lie{g}})$.
	Then $\Rrankmax(V/\lie{u}V) = n$ holds, and for any $Q \in \Ass_{\univ{t}}(V/\lie{u}V)$ with $\dim(\Variety(Q)) = n$,
	there exists $P \in \Ass_{\univcent{g}}(V)$ such that $\Variety(P) = q(\Variety(Q) - \rho)$.
\end{fact}

We will show in Corollary \ref{cor:FiniteLenUcoinv} that a larger algebra acts on $V/\lie{u}V$ such that the module has finite length if $V$ comes from a holonomic $\ntDsheaf{}$-module.
To construct the algebra action, we use the basic affine space $G/U$.

We shall consider the case that $G$ is commutative.
Let $(\alg{A}, T)$ be a generalized pair with complex torus $T$.
In this case, an irreducible $\alg{A}$-module has a good decomposition like weight decomposition, which is proved in \cite[Theorem 3.33 and Corollary 3.34]{Ki24-2}.
The results have been proved as a special case of the general case for non-abelian groups.
We give a simple proof based on the Jacobson density theorem to be self-contained.

\begin{lemma}\label{lem:Affinity}
	Let $V$ be an irreducible $\alg{A}$-module.
	Then $V = \bigoplus_{P \in \Ass_{\univ{t}}(V)} V^P$ holds, and
	for any $P, Q \in \Ass_{\univ{t}}(V)$, there exists a weight $\lambda$ in $\alg{A}$ such that $\Variety(Q) = \Variety(P) + \lambda$.
\end{lemma}

\begin{proof}
	Fix $P \in \Ass_{\univ{t}}(V)$.
	It is easy to see that the variety $\Variety(\Ann_{\univ{t}}(Xv))$ is contained in a translation of $\Variety(P)$ for any $v \in V^P$ and weight vector $X \in \alg{A}$.
	From this and the irreducibility of $V$, the dimension of $\Variety(\Ann_{\univ{t}}(v))$ does not depend on $0 \neq v\in V$,
	and hence $\Variety(Q)$ is a translation of $\Variety(P)$ for any $Q \in \Ass_{\univ{t}}(V)$.
	In other words, for any weight $\lambda$ in $\alg{A}$, there is $Q \in \Ass_{\univ{t}}(V)$ such that $\alg{A}^\lambda \cdot V^P \subset V^Q$.
	Hence we have $\sum_{Q \in \Ass_{\univ{t}}(V)} V^Q \supset \alg{A}V^P = V$.

	Since $\dim(\Variety(Q))$ is independent of $Q \in \Ass_{\univ{t}}(V)$, we have
	\begin{align*}
		(V^{Q_1} + V^{Q_2} + \cdots + V^{Q_r}) \cap V^{Q_{r+1}} \subset V^{(Q_1 \cap Q_2 \cap \cdots \cap Q_r) + Q_{r+1}} = 0
	\end{align*}
	for any mutually different $Q_1, \ldots, Q_{r+1} \in \Ass_{\univ{t}}(V)$.
	Hence $\sum_{Q \in \Ass_{\univ{t}}(V)} V^Q$ is a direct sum.
	We have shown the lemma.
\end{proof}

\begin{theorem}\label{thm:Affinity}
	Let $V$ be an irreducible $\alg{A}$-module such that $\End_{\alg{A}}(V) = \CC$.
	Then there exists a connected closed subgroup $M\subset T$ satisfying the following properties.
	\begin{enumerate}
		\item $V = \bigoplus_{P \in \Ass_{\univ{t}}(V)} V^P$ is the weight space decomposition with respect to $\lie{m}$.
		\item For each $P \in \Ass_{\univ{t}}(V)$, $\Variety(P)\subset \lie{t}^*$ is a translation of $(\lie{t}/\lie{m})^*$.
		\item Each $V^P$ is an irreducible $\alg{A}^\lie{m}$-module.
		\item For each weights $\lambda, \mu \in \lie{m}^*$ in $V$, we have $\alg{A}^{\lambda - \mu}\cdot V^\mu = V^\lambda$.
	\end{enumerate}
\end{theorem}

\begin{remark}\label{remark:Schur}
	The assumption $\End_{\alg{A}}(V) = \CC$ holds if $V$ has at most countable dimension by Schur's lemma.
\end{remark}

\begin{proof}
	We shall construct the subalgebra $\lie{m}$.
	Fix $P \in \Ass_{\univ{t}}(V)$.
	Set $R\coloneq \set{\lambda \in \lie{t}^* : \alg{A}^\lambda\neq 0, \Variety(P) + \lambda = \Variety(P)}$
	and $\alg{B}\coloneq \bigoplus_{\lambda \in R} \alg{A}^\lambda$.
	Then $\alg{B}$ is a subalgebra of $\alg{A}$, and $(\alg{B}, T)$ forms a generalized pair.
	By Lemma \ref{lem:Affinity} and its proof, $V^P$ is an irreducible $\alg{B}$-module.
	By the assumption $\End_{\alg{A}}(V) = \CC$, Lemma \ref{lem:Affinity} and the Jacobson density theorem, we have $\End_{\alg{B}}(V^P) = \CC$.

	Set $\lie{m}\coloneq \bigcap_{\lambda \in R} \Ker(\lambda)$
	and write $\alpha \colon \lie{t}^* \rightarrow \lie{m}^*$ for the restriction map.
	Then $\lie{m}$ is a Lie algebra of some connected closed subgroup $M$ of $T$.
	By definition, the image of $\lie{m}$ is contained in the center of $\alg{B}$.
	Hence $\lie{m}$ acts on $V^P$ by a character $\mu \in \lie{m}^*$
	and we have $\Variety(P)\subset \alpha^{-1}(\mu)$.
	Note that $\alpha^{-1}(\mu)$ is a translation of $(\lie{t}/\lie{m})^* = \mathrm{span}_{\CC} R$.
	By the definition of $R$ and $\lie{m}$, we have $\Variety(P) + (\lie{t}/\lie{m})^* = \Variety(P)$.
	Therefore, we obtain $\Variety(P) = \alpha^{-1}(\mu)$.

	The assertions follow easily from the construction of $\lie{m}$.
\end{proof}

Precisely speaking, we do not need the Jacobson density theorem.
In fact, it is easy to see that $\alg{A}\otimes_{\alg{B}}V^P$ has a unique irreducible quotient, which is isomorphic to $V$.
This gives another proof of $\End_{\alg{B}}(V^P) \simeq \End_{\alg{A}}(V) = \CC$.
Our proof is based on the observation that $\alg{A}$ is enough large in $\End_{\CC}(V)$ and has enough many weights to recover $\Variety(P)$, which we can read from the Jacobson density theorem.

It is easy to see that $\lie{m}$ in Theorem \ref{thm:Affinity} is uniquely determined by $V$
and $\lie{t}/\lie{m}$ is a unique small Cartan subalgebra for $V$.
Note that if a $\lie{t}$-module $V$ has a small Cartan subalgebra, so does its submodule.
Using this, we give a definition of small Cartan subalgebras for $\ntDsheaf{}$-modules.
Let $X$ be a connected smooth $T$-variety and $\Dsheaf{A}{X}$ a $T$-equivariant TDO on $X$.

\begin{proposition}\label{prop:Localness}
	Let $\Dmod{M}$ be an irreducible quasi-coherent $\Dsheaf{A}{X}$-module and $V$ a $T$-stable open subset of $X$.
	Suppose that $\sect(V, \Dmod{M}) \neq 0$.
	Then $\sect(V, \Dmod{M})|_{\lie{t}}$ has a small Cartan subalgebra, which does not depend on the choice of $V$.
\end{proposition}

\begin{proof}
	Note that any quasi-projective smooth $T$-variety is covered by $T$-stable affine open subsets from a local structure theorem.
	See \cite[Theorem 4.7]{Ti11}.
	Since any coherent submodule of $\Dmod{M}|_V$ extends to a coherent submodule of $\Dmod{M}$
	(see \cite[Corollary 1.4.17]{HTT08}), $\sect(V', \Dmod{M})$ is zero or irreducible for any affine open subset $V' \subset X$.

	Since $\Dmod{M}$ is irreducible, the support $S$ of $\Dmod{M}$ is an irreducible subvariety
	and the support of any non-zero local section of $\Dmod{M}$ is an open subset of $S$.
	Let $V'$ be a $T$-stable affine open subset of $V$ with $V'\cap S \neq \emptyset$.
	Then $\sect(V', \Dmod{M})$ is irreducible, and $\sect(V, \Dmod{M})$ is a $\lie{t}$-submodule of $\sect(V', \Dmod{M})$.
	By Theorem \ref{thm:Affinity}, $\sect(V', \Dmod{M})$ has a small Cartan subalgebra $\lie{t}/\lie{m}$.
	Hence $\lie{t}/\lie{m}$ is also a small Cartan subalgebra for $\sect(V, \Dmod{M})$.
	This implies that the small Cartan subalgebra does not depend on the choice of $V'$ and hence $V$.
\end{proof}

\begin{remark}
	A Cartan subalgebra may depend on the choice of the $T$-stable open subset $V$.
\end{remark}

\begin{definition}\label{def:CartanDmod}
	Let $\Dmod{M}$ be a non-zero $\Dsheaf{A}{X}$-module.
	If $\sect(V, \Dmod{M})|_{\lie{t}}$ has a small Cartan subalgebra for any $T$-stable open subset $V$ of $X$ with $\sect(V, \Dmod{M})\neq 0$ and it does not depend on the choice of $V$, we say that the algebra is a \define{small Cartan subalgebra} for $\Dmod{M}$.
\end{definition}

\subsection{Universal module}\label{subsection:G/Umodules}

In this subsection, we construct a $\ntDsheaf{G/U}$-module related to Verma modules.
We will use the module to construct $V/\lie{u}V$ by $\ntDsheaf{}$-modules.

Let $G$ be a simply-connected connected semisimple algebraic group over $\CC$.
Fix a Borel subgroup $B=TU$ of $G$ with a maximal torus $T$ and unipotent radical $U$.

Since the $U$-action on $\wedge^{top}(\lie{g}/\lie{u})^*$ is trivial,
there is a unique non-zero $G$-invariant top form $\omega \in \sect(\Omega_{G/U})$ up to scalar.
The isomorphism $\rsheaf{G/U} \xrightarrow{\sim}\Omega_{G/U}$ ($f \mapsto f\omega$)
induces an isomorphism $\varphi\colon \ntDsheaf{G/U}\xrightarrow{\sim} \ntDsheaf{G/U}^{op}$.
By the definition of the right $\ntDsheaf{G/U}$-action on $\Omega_{G/U}$ (see \cite[Lemma 1.2.5]{HTT08}),
we have
\begin{align*}
	\varphi(L(X)) &= -L(X) \\
	\varphi(R(H)) &= -R(H) - 2\rho(H)
\end{align*}
for any $X \in \lie{g}$ and $H \in \lie{t}$.

Let $\uVerm$ denote the unique irreducible holonomic $\ntDsheaf{G/U}$-module supported on $\set{eU}$.
We describe the $(\lie{g\oplus t})$-action on $\uVerm$.
Identify $\CC \otimes_{\univ{u}}\univ{g}$ with the set of distributions on a real form of $G/U$
supported on $\set{eU}$.
Then $\CC \otimes_{\univ{u}}\univ{g}$ admits a right $\ntDsheaf{G/U}$-module structure and hence a left $\ntDsheaf{G/U}^\opalg$-module structure.
Using the isomorphism $\varphi$ and applying the canonical anti-automorphism ${}^t$ of $\univ{g}$, we obtain a left $\ntDsheaf{G/U}$-module $\univ{g}\otimes_{\univ{u}}\CC \otimes \CC_{-2\rho}$ with $(\lie{g\oplus t})$-action given by
\begin{align*}
	L(X)\cdot Z\otimes 1 \otimes 1 &= XZ \otimes 1 \otimes 1 \\
	R(H)\cdot Z\otimes 1 \otimes 1 &= -XH \otimes 1 \otimes 1 - 2\rho(H) Z\otimes 1\otimes 1
\end{align*}
for $X \in \lie{g}, H \in \lie{t}$ and $Z \in \univ{g}$.
Then $\uVerm$ is isomorphic to $\univ{g}\otimes_{\univ{u}}\CC \otimes \CC_{-2\rho}$ as a $\ntDsheaf{G/U}$-module.
Note that the twist $\CC_{-2\rho}$ is necessary so that the $\univ{g}\otimes \univ{t}$-action factors through $\univ{g}\otimes_{\univcent{g}}\univ{t}$.
See \eqref{eqn:Waction} for the homomorphism $\univcent{g}\rightarrow \univ{t}$.

\begin{proposition}\label{prop:AcyclicUVerma}
	$\uVerm$ is acyclic and $\sect(\uVerm) \simeq \univ{g}\otimes_{\univ{u}}\CC \otimes \CC_{-2\rho}$ as $(\lie{g\oplus t})$-modules.
\end{proposition}

\begin{proof}
	Write $\alpha \colon G/U\rightarrow \overline{G/U}$ for the inclusion.
	Fix an affine open subset $V\subset G/U$ containing $eU$, and let $\beta \colon V\rightarrow G/U$ denote the inclusion.
	Since the support of $\uVerm$ is contained in $V$, we have $\uVerm \simeq \beta_* \beta^* \uVerm$.
	Hence we obtain
	\begin{align*}
		H^i(G/U, \uVerm) &\simeq H^i \circ R\sect \circ \alpha_* \beta_* \beta^* \uVerm \\
		&\simeq H^i(V, \beta^*\uVerm) \simeq \begin{cases}
			\univ{g}\otimes_{\univ{u}}\CC \otimes \CC_{-2\rho} & (i = 0), \\
			0 & (i \neq 0).
		\end{cases}
	\end{align*}
	We used that $\beta^* = \beta^{-1}$ is exact and $V$ is affine.
\end{proof}

The module $\uVerm$ can be regarded as a family of Verma modules as follows.
For $w \in W_G$, we set $\uVerm_{w}\coloneq \loc(\sect(\uVerm)^w)$,
where $(\cdot)^w$ is the Fourier transform defined in Subsection \ref{subsection:reviewFourier}.
By the construction of $\uVerm_{w}$, we obtain the following proposition.

\begin{proposition}\label{prop:UniversalVerma}
	Let $\lambda\in \lie{t}^*$.
	Take $w \in W_G$ such that $w(\lambda)$ is anti-dominant, and set $u\coloneq w_lww_l^{-1}$.
	Then there exists a $\lie{g}$-module isomorphism
	\begin{align*}
		H^0\circ R\sect(Dp_{+, w(\lambda)} (\uVerm_{u^{-1}})) \simeq \univ{g}\otimes_{\univ{b}}\CC_{\lambda-\rho},
	\end{align*}
	where $Dp_{+,w(\lambda)}$ is the direct image functor from $D^b_{h}(\ntDsheaf{G/U})$ to $D^b_{h}(\ntDsheaf{G/B, w(\lambda)})$.
\end{proposition}

\begin{proof}
	By Corollary \ref{cor:localizationCoinv}, we have
	\begin{align*}
		H^0\circ R\sect(Dp_{+, w(\lambda)} (\uVerm_{u^{-1}})) \simeq \CC_{w(\lambda)+\rho}\otimes_{\univ{t}}\sect(\uVerm)^{u^{-1}}.
	\end{align*}
	The $(\lie{g\oplus t})$-action on $\sect(\uVerm)^{u^{-1}}$ can be computed by Fact \ref{fact:fourier} and $\uVerm \simeq \univ{g}\otimes_{\univ{u}}\CC \otimes \CC_{-2\rho}$.
	Hence it is easy to see $\CC_{w(\lambda)+\rho}\otimes_{\univ{t}}\sect(\uVerm)^{u^{-1}} \simeq \univ{g}\otimes_{\univ{b}}\CC_{\lambda-\rho}$.
\end{proof}

\subsection{Existence of large action}\label{subsection:LargeAction}

Retain the notation in the previous subsection.
Let $X$ be a smooth $G$-variety and $\Dsheaf{A}{X}$ a $G$-equivariant TDO on $X$.
Suppose that the global section functor $\sect$ is exact on $\Mod_{qc}(\Dsheaf{A}{X})$.
We have defined the holonomicity of $\Dalg{A}{X}\otimes \ntDalg{G/U}$-modules in Definition \ref{def:holonomic2}.

Let $M \in \Mod(\Dalg{A}{X})$ and $N \in \Mod(\ntDalg{G/U})$.
Then the $\CC$-algebra $(\Dalg{A}{X}\otimes \ntDalg{G/U})^G$ acts on $M\otimes_{\univ{g}} N$.
We shall consider the length of $M\otimes_{\univ{g}} N$
and more generally the cohomologies $H^i(\lie{g}; M\otimes N)$.
To reduce the problem to that of $\Dalg{A}{X}\otimes \ntDalg{G/U}$-modules, we use the Zuckerman derived functor.
Recall the isomorphism stated in Fact \ref{fact:GInvZuckerman}:
\begin{align*}
	\Dzuck{G}{}{i}(M\otimes N)^G \simeq H^i(\lie{g}; M\otimes N).
\end{align*}
Note that $\Dzuck{G}{}{i}(M\otimes N)$ is an $(\Dalg{A}{X}\otimes \ntDalg{G/U}, G)$-module.
Since the $G$-action on $\Dzuck{G}{}{i}(M\otimes N)$ is completely reducible, 
the length of $H^i(\lie{g}; M\otimes N)$ is bounded by that of $\Dzuck{G}{}{i}(M\otimes N)$ from above.

\begin{theorem}\label{thm:FiniteLenGeneral}
	Let $\Dmod{M}\in \Mod_{h}(\Dsheaf{A}{X})$ and $\Dmod{N} \in \Mod_{h}(\ntDsheaf{G/U})$.
	Suppose that $\Dmod{N}$ is acyclic.
	Then the cohomologies $H^i(\lie{g}; \sect(\Dmod{M})\otimes \sect(\Dmod{N}))$ are $(\Dalg{A}{X}\otimes \ntDalg{G/U})^G$-modules of finite length.
\end{theorem}

\begin{proof}
	As stated above, it is enough to show that $\Dzuck{G}{}{i}(\sect(\Dmod{M})\otimes \sect(\Dmod{N}))$'s have finite length.
	By assumption, $\Dmod{M}\boxtimes \Dmod{N}$ is acyclic and holonomic.
	See \cite[Lemma 1.5.31]{HTT08} for the acyclicity.
	Since the functors $\DlocZuck{G}{}{q}$ are compositions of the direct image functor and the inverse image functor,
	all the modules $\DlocZuck{G}{}{q}(\Dmod{M}\boxtimes \Dmod{N})$ are holonomic, and by Theorem \ref{thm:HolonomicGlobalAX}, all the cohomologies $H^p(X, \DlocZuck{G}{}{q}(\Dmod{M}\boxtimes \Dmod{N}))$ have finite length.
	The assertion therefore follows from Corollary \ref{cor:FiniteZuckerman}.
\end{proof}

In the previous subsection, we have constructed $\uVerm$.
Applying Theorem \ref{thm:FiniteLenGeneral} to $\Dmod{N} = \uVerm$, we obtain the following corollary.

\begin{corollary}\label{cor:FiniteLenUcoinv}
	Let $\Dmod{M}\in \Mod_{h}(\Dsheaf{A}{X})$.
	Then the $\lie{t}$-module $H^i(\lie{u}; \sect(\Dmod{M}))$ extends to an $(\Dalg{A}{X}\otimes \ntDalg{G/U})^G$-module of finite length for any $i \in \NN$.
\end{corollary}

\begin{proof}
	By the Poincar\'e duality (see \cite[Corollary 3.6]{HTT08}), we have
	\begin{align*}
		&H^i(\lie{g}; \sect(\Dmod{M})\otimes (\univ{g}\otimes_{\univ{u}}\CC\otimes \CC_{-2\rho})) \\
		&\simeq H_{n-i}(\lie{g}; \sect(\Dmod{M})\otimes (\univ{g}\otimes_{\univ{u}}\CC\otimes \CC_{-2\rho})) \\
		&\simeq H_{n-i}(\lie{u}; \sect(\Dmod{M})\otimes \wedge^{m} \lie{u}^*) \\
		&\simeq H^{m+i-n}(\lie{u}; \sect(\Dmod{M})),
	\end{align*}
	where $n = \dim(\lie{g}), m = \dim(\lie{u})$.
	Since $\uVerm$ is acyclic and $\sect(\uVerm) \simeq \univ{g}\otimes_{\univ{u}}\CC\otimes \CC_{-2\rho}$
	by Proposition \ref{prop:AcyclicUVerma}, the assertion follows from Theorem \ref{thm:FiniteLenGeneral}.
\end{proof}

We shall consider the $\lie{t}$-actions on the cohomology groups $H^i(\lie{u}; \sect(\Dmod{M}))$.
As proved in Corollary \ref{cor:FiniteLenUcoinv}, $(\Dalg{A}{X}\otimes \ntDalg{G/U})^G$ acts on the cohomologies.
Combining Theorem \ref{thm:Affinity} and Corollary \ref{cor:FiniteLenUcoinv}, we obtain the following result.

\begin{theorem}\label{thm:AffinityUcohomology}
	Let $i \in \NN$ and $\Dmod{M}\in \Mod_{h}(\Dsheaf{A}{X})$.
	Then there exists a $\lie{t}$-module finite exhaustive filtration $0 = V_0\subset V_1 \subset \cdots \subset V_r$
	of $H^i(\lie{u}; \sect(\Dmod{M}))$ such that each $V_j / V_{j-1}$ has a small Cartan subalgebra.
\end{theorem}

\begin{proof}
	A composition series of $H^i(\lie{u}; \sect(\Dmod{M}))$ as an $(\Dalg{A}{X}\otimes \ntDalg{G/U})^G$-module satisfies the desired condition
	by Theorem \ref{thm:Affinity}.
	The assumption $\End_{\alg{A}}(V) = \CC$ is fulfilled because $(\Dalg{A}{X}\otimes \ntDalg{G/U})^G$ has at most countable dimension.
	See Remark \ref{remark:Schur}.
\end{proof}

\subsection{Branching problem}\label{subsection:Branching}

In this subsection, we deal with the branching problem of representations of reductive Lie groups.
Although the result in this subsection are direct consequences of the general result (Theorem \ref{thm:AffinityUcohomology}),
we state it explicitly.

\newcommand{\LG}[1]{\widetilde{#1}}
\newcommand{\SM}[1]{#1}

Let $\LG{G}$ be a connected reductive algebraic group and $\SM{G}$ a connected reductive subgroup of $\LG{G}$.
Fix a Borel subgroup $\LG{B}=\LG{T}\LG{U}$ of $\LG{G}$ with a maximal torus $\LG{T}$ and unipotent radical $\LG{U}$.
Similarly, take a Borel subgroup $\SM{B} = \SM{T}\SM{U}$ of $\SM{G}$.
Let $\proots$ denote the set of positive roots of $\lie{\LG{g}}$ determined by $\LG{B}$.
Write $\rho$ for half the sum of elements in $\proots$,
and $p\colon \LG{G}/\LG{U}\rightarrow \LG{G}/\LG{B}$ for the natural projection.
For $\lambda \in \lie{\LG{t}}^*$, we set
\begin{align*}
	\ntDsheaf{\LG{G}/\LG{B},\lambda}\coloneq (p_* \ntDsheaf{\LG{G}/\LG{U}})^{\LG{T}}\otimes_{\univ{\LG{t}}} \End_{\CC}(\CC_{\lambda + \rho}).
\end{align*}
Then $\ntDsheaf{\LG{G}/\LG{B},\lambda}$ is a $\LG{G}$-equivariant TDO on $\LG{G}/\LG{B}$.
See Subsection \ref{subsection:BB} for the TDO.

As Definition \ref{def:holonomic2}, we shall define holonomicity of $\lie{\LG{g}}$-modules.
Recall the Beilinson--Bernstein correspondence (Fact \ref{fact:beilinsonBernsten}).

\begin{definition}
	We say that an irreducible $\lie{\LG{g}}$-module $V$ is \define{holonomic} if there exists $\Dmod{M} \in \Mod_{h}(\ntDsheaf{\LG{G}/\LG{B},\lambda})$ such that $\lambda$ is anti-dominant and $\sect(\Dmod{M}) \simeq V$.
	We say that a $\lie{\LG{g}}$-module $V$ is \define{holonomic} if $V$ has finite length and all the composition factors of $V$ are holonomic.
\end{definition}

For example, Harish-Chandra modules and objects in the BGG category $\mathcal{O}$ are holonomic.
See \cite[Theorem 11.6.1]{HTT08}.

Combining Corollary \ref{cor:FiniteLenUcoinv} and the Beilinson--Bernstein correspondence, we obtain the following theorem.
Although $G$ may not be simply-connected semisimple, it is easy to reduce this case to the previous one by taking a finite covering.
Note that the abelian factors of $G$ does not affect the results about $G/U$.

\begin{theorem}
	Let $V$ be a holonomic $\lie{\LG{g}}$-module (e.g.\ Harish-Chandra modules and highest weight modules) and $i \in \NN$.
	Then there exists a $\lie{t}$-module finite exhaustive filtration $0 = V_0\subset V_1 \subset \cdots \subset V_r$
	of $H^i(\lie{u}; V)$ such that each $V_j/V_{j-1}$ has a small Cartan subalgebra.
\end{theorem}

In the case of Harish-Chandra modules, the problem to find a small Cartan subalgebra for $V$
can be reduced to that for the $(\univ{\LG{g}}\otimes \ntDalg{G/U})^G$-module $V/\lie{u}V\otimes \CC_{-2\rho} \simeq H^{n}(\lie{u}, V)$, where $n = \dim(\lie{u})$.
Let $K$ be a (connected and finite) covering of a symmetric subgroup of $G$.
Fact \ref{fact:ReductionToUcoinv} and Corollary \ref{cor:FiniteLenUcoinv} show the following theorem.

\begin{theorem}
	Let $V$ be an irreducible holonomic $(\lie{\LG{g}}, K)$-module.
	Then there exists an irreducible $(\univ{\LG{g}}\otimes \ntDalg{G/U})^G$-module subquotient $V'$ of $V/\lie{u}V\otimes \CC_{-2\rho}$
	such that the small Cartan subalgebra for $V'$ is one for $V$.
\end{theorem}

\bibliographystyle{abbrv}


\end{document}